\documentclass[10pt,twocolumn,twoside]{IEEEtran} 

\usepackage{amsmath,amssymb,mathtools,amsthm}
\usepackage{graphicx,color}
\usepackage{cite}

\newtheorem{theorem}{Theorem}[section]
\newtheorem{proposition}[theorem]{Proposition}
\newtheorem{lemma}[theorem]{Lemma}
\newtheorem{corollary}[theorem]{Corollary}
\newtheorem{remark}[theorem]{Remark}
\newtheorem{example}[theorem]{Example}
\newtheorem{definition}[theorem]{Definition}
\newtheorem{problem}{Problem}

\newcommand{\longthmtitle}[1]{\mbox{}\textit{(#1):}}

\newcommand{\real}{\ensuremath{\mathbb{R}}}
\newcommand{\realpos}{\ensuremath{\mathbb{R}_{>0}}}

\newcommand{\intpos}{{\mathbb{N}}}
\newcommand{\intnonneg}{{\mathbb{Z}}_{\ge 0}}

\DeclareMathOperator*{\argmin}{argmin}
\DeclareMathOperator*{\minimize}{minimize}

\let\emptyset\varnothing

\newcommand{\setdef}[2]{\{#1 \; | \; #2\}}
\newcommand{\setdefb}[2]{\big\{#1 \; | \; #2\big\}}
\newcommand{\setdefB}[2]{\Big\{#1 \; | \; #2\Big\}}

\newcommand{\Fc}{\mathcal{F}}
\newcommand{\Gc}{\mathcal{G}}
\newcommand{\Hc}{\mathcal{H}}
\newcommand{\Ic}{\mathcal{I}}

\newcommand{\Kc}{\mathcal{K}}
\newcommand{\Lc}{\mathcal{L}}
\newcommand{\Mc}{\mathcal{M}}
\newcommand{\Nc}{\mathcal{N}}
\newcommand{\Oc}{\mathcal{O}}
\newcommand{\Pc}{\mathcal{P}}

\newcommand{\Sc}{\mathcal{S}}

\newcommand{\Vc}{\mathcal{V}}
\newcommand{\Wc}{\mathcal{W}}
\newcommand{\Xc}{\mathcal{X}}

\newcommand{\Lap}{\mathrm{Lap}}

\newcommand{\pr}{\mathbb{P}}
\newcommand{\E}{\mathbb{E}}

\newcommand{\until}[1]{\{1,\dots,#1\}}
\newcommand{\proj}{\text{proj}}

\newcommand{\erf}{\text{erf}}

\newcommand{\oprocendsymbol}{\hbox{$\square$}}
\newcommand{\oprocend}{\relax\ifmmode\else\unskip\hfill\fi\oprocendsymbol}
\def\eqoprocend{\tag*{$\square$}}

\graphicspath{{figures/}}

\title{Differentially Private Distributed Convex Optimization via
  Functional Perturbation\thanks{A preliminary version of this paper
    has been submitted to the 2016 American Control Conference,
    Boston, MA}}

\author{Erfan Nozari \quad Pavankumar Tallapragada \quad Jorge
  Cort\'es%
  \thanks{The authors are with the Department of Mechanical and
    Aerospace Engineering, University of California, San Diego,
    \{enozari,ptallapragada,cortes\}@ucsd.edu}}

\begin{document}

\maketitle
\thispagestyle{empty}
\pagestyle{empty}

\begin{abstract}
  We study a class of distributed convex constrained optimization
  problems where a group of agents aim to minimize the sum of
  individual objective functions while each desires that
    any information about its objective function is kept private.  We
  prove the impossibility of achieving differential privacy using
  strategies based on perturbing the inter-agent messages with noise
  when the underlying noise-free dynamics are
  asymptotically stable.  This justifies our algorithmic solution
  based on the perturbation of individual functions with Laplace
  noise. To this end, we establish a general framework for
  differentially private handling of functional data.  We further
  design post-processing steps that ensure the perturbed functions
  regain the smoothness and convexity properties of the original
  functions while preserving the differentially private guarantees of
  the functional perturbation step. This methodology allows us to use
  any distributed coordination algorithm to solve the optimization
  problem on the noisy functions.  Finally, we explicitly bound the
  magnitude of the expected distance between the perturbed and true
  optimizers which leads to an upper bound on the
    privacy-accuracy trade-off curve. Simulations illustrate our
  results.
\end{abstract}


\section{Introduction}\label{sec:intro}

Privacy preservation is an increasingly critical issue that plays a
key role in preventing catastrophic failures in physical
infrastructure as well as easing the social adoption of new
technology.  Power networks, manufacturing systems, and smart
transportation are just but a few examples of cyberphysical
applications in need of privacy-aware design of control and
coordination strategies.  In these scenarios, the problem of
optimizing the operation of a group of networked resources is a common
and important task, where the individual objective functions
associated to the entities, the estimates of the optimizer, or even
the constraints on the optimization might reveal sensitive
information.  Our work here is motivated by the goal of synthesizing
distributed coordination algorithms that solve networked optimization
problems with privacy guarantees \emph{and} high accuracy.

\subsubsection*{Literature review}
Our work builds upon the existing literature of distributed convex
optimization and differential privacy. In the area of networked
systems, an increasing body of research,
e.g.,~\cite{DPB-JNT:97,AN-AO-PAP:10,BJ-MR-MJ:09,MZ-SM:12,BG-JC:14-tac,JCD-AA-MJW:12}
and references therein, designs and analyzes algorithms for
distributed convex optimization both in discrete and continuous time
as well as in deterministic and stochastic scenarios. While these
works consider an ambitious suite of topics related to convergence and
performance under various constraints imposed by real-world
applications, privacy is an aspect generally absent in their
treatment.  The concept of differential
privacy~\cite{CD-FM-KN-AS:06,CD:06} was originally proposed for
databases of individual records subject to public queries and has been
extended to several areas thereafter. The recent work~\cite{CD-AR:14}
provides a comprehensive recent account of this area.  In machine
learning, the problem of differentially private optimization has
received attention, see
e.g.~\cite{KC-CM-ADS:11,DK-AS-AT:12,JZ-ZZ-XX-YY-MW:12,RH-AR-LW:13,AR-SA:12,SS-KC-ADS:13,KC-DH-SS:14},
as an intermediate, usually centralized, step for solving other
learning or statistical tasks. The common paradigm is having the
sensitive information correspond to the entries of a finite database
of records or training data that usually constitute the parameters of
an additive objective function. Threat models are varied, including
releasing to the adversary the whole sequence of internal states of
the optimization process or only the algorithm's final
output. In~\cite{KC-CM-ADS:11}, the authors design a differentially
private classifier by perturbing the objective function with a linear
finite-dimensional function (hyper-plane).  It is shown
in~\cite{DK-AS-AT:12} that this method also works in the presence of
constraints and non-differentiable regularizers.  Although this is
sufficient to preserve the privacy of the underlying
finite-dimensional parameter set (learning samples), it cannot keep
the whole objective functions private.  The authors
of~\cite{JZ-ZZ-XX-YY-MW:12} design a sensitivity-based differentially
private algorithm for regression analysis which, instead of perturbing
the optimal weight vector, perturbs the regression cost function by
injecting noise into the coefficients of the quadratic truncation of
its Taylor expansion.  This truncation limits the functional space to
the (finite-dimensional) space of quadratic
functions. In~\cite{RH-AR-LW:13}, the authors propose the addition of
a sample path of a Gaussian random process to the objective function,
but do not explore the generalization to arbitrary dimensions or
ensure the smoothness and convexity of the resulting function.  In
general, the proposed algorithms are not distributed and neither
designed for nor capable of preserving the privacy of
infinite-dimensional objective functions.  Furthermore, the work in
this area does not rigorously study the effect of added noise on the
global optimizer or on the smoothness and convexity properties of the
objective functions. In addition to addressing these issues, the
present treatment is applicable to scenarios where the sensitive
information consists of elements of any separable Hilbert space,
including (objective) functions coming from the (infinite-dimensional)
$L_2$ space.

Of more relevance to our paper are recent
works~\cite{ZH-SM-NV:15,SH-UT-GJP:16,MTH-ME:15} that are focused on
differentially private distributed optimization problems for
multi-agent systems.  These papers consider as private information,
respectively, the objective functions, the optimization constraints,
and the agents' states. The underlying commonality is the algorithm
design approach based on the idea of message perturbation. This idea
consists of adopting a standard distributed optimization algorithm and
modifying it by having agents perturb the messages to their neighbors
or a central coordinator with Laplace or Gaussian noise.  This
approach has the advantage of working with the original objective
functions and thus is easy to implement. However, for fixed design
parameters, the algorithm's output does not correspond to the true
optimizer in the absence of noise, suggesting the presence of a
steady-state accuracy error.  This problem is addressed
in~\cite{SH-UT-GJP:16} by terminating the algorithm after a finite
number of steps and optimizing this number offline as a function of
the desired level of privacy.  Nevertheless, for any fixed level of
privacy, there exists an amount of bias in the algorithm's output
which is not due to the added noise but to the lack of asymptotic
stability of the underlying noiseless dynamics.  To address this
issue, our approach explores the use of functional perturbation to
achieve differential privacy. The concept of functional differential
privacy combines the benefits of metrics and adjacency relations. The
authors of~\cite{KC-MEA-NEB-CP:13} also employ metrics instead of
binary adjacency relations in the context of differential
privacy. This approach has the advantage that the difference between
the probabilities of events corresponding to any pair of data sets is
bounded by a function of the distance between the data sets,
eliminating the need for the computation of conservative sensitivity
bounds.

\subsubsection*{Statement of contributions}
We consider a group of agents that seek to minimize the sum of their
individual objective functions over a communication network in a
differentially private manner.  Our first contribution is to show that
coordination algorithms which rely on perturbing the agents' messages
with noise cannot satisfy the requirements of differential privacy if
the underlying noiseless dynamics are locally asymptotically stable.
The presence of noise necessary to ensure differential privacy is
known to affect the algorithm accuracy in solving the distributed
convex optimization problem. However, this result explains why
message-perturbing strategies incur additional inaccuracies that are
present even if no noise is added.  Our second contribution is
motivated by the goal of guaranteeing that the algorithm accuracy is
only affected by the presence of noise. We propose a general framework
for functional differential privacy over Hilbert spaces and introduce
a novel definition of adjacency using adjacency spaces. The latter
notion is quite flexible and includes, as a special case, the
conventional bounded-difference notion of adjacency. We carefully
specify these adjacency spaces within the $L_2$ space such that the
requirement of differential privacy can be satisfied with bounded
perturbations.  Our third contribution builds on these results on
functional perturbation to design a class of distributed,
differentially private coordination algorithms. We let each agent
perturb its own objective function based on its desired level of
privacy, and then the group uses any provably correct distributed
coordination algorithm to optimize the sum of the individual perturbed
functions.  Two challenges arise to successfully apply this strategy:
the fact that the perturbed functions might lose the smoothness and
convexity properties of the original functions and the need to
characterize the effect of the added noise on the minimizer of the
resulting problem. We address the first challenge using a cascade of
smoothening and projection steps that maintain the differential
privacy of the functional perturbation step. We address the second
challenge by explicitly bounding the absolute expected deviation from
the original optimizer using a novel Lipschitz characterization of the
$\argmin$ map. By construction, the resulting coordination algorithms
satisfy the requirement of recovering perfect accuracy in the absence
of noise. Various simulations illustrate our results.

\subsubsection*{Organization}

We introduce our notation and basic preliminaries in
Section~\ref{sec:prelims} and formulate the private distributed
optimization problem in
Section~\ref{sec:prob-state}. Section~\ref{sec:rationale} presents the
rationale for our design strategy and
Section~\ref{sec:func-diff-privacy} describes a general framework for
functional differential privacy.  We formulate our solution to the
private distributed optimization problem in
Section~\ref{sec:diff-private-dist-opt}.  We present simulations in
Section~\ref{sec:sims} and collect our conclusions and ideas for
future work in
Section~\ref{sec:conclusion}. Appendix~\ref{sec:argmin-lip} gathers
our results on the Lipschitzness of the $\argmin$ map under suitable
assumptions.

\section{Preliminaries}\label{sec:prelims}

In this section, we introduce our notational conventions and some
fundamental facts about Hilbert spaces and robust stability of
discrete-time systems.

\subsection{Notation}
We use $\real$, $\realpos$, $\intnonneg$, and $\intpos$ to denote the
set of reals, positive reals, nonnegative integers, and positive
integers, respectively.  The space of scalar- and $n$-vector-valued
infinite sequences are denoted by $\real^\intpos$ and
$(\real^n)^\intpos$, respectively. Given $K \in \intpos$ and an
element $\boldsymbol{\eta} = \{\eta(k)\}_{k = 0}^\infty$ of
$\real^\intpos$ or $(\real^n)^\intpos$, we use the shorthand notation
$\boldsymbol{\eta}_K = \{\eta(k)\}_{k = 0}^K$. If the index of
$\boldsymbol{\eta}$ starts at $k=1$, with a slight abuse of notation
we also denote $\{\eta(k)\}_{k = 1}^K$ by $\boldsymbol{\eta}_K$.  We
denote by $\ell_2 \subset \real^\intpos$ the space of square-summable
infinite sequences.  We use $| \cdot |_p$ and $\| \cdot \|_p$ for the
$p$-norm in finite and infinite-dimensional normed vector spaces,
respectively (we drop the index $p$ for $p = 2$). We let $B(c, r)$
denote the closed ball with center $c$ and radius $r$ in Euclidean
space.
For $D \subseteq \real^d$, $D^o$ denotes its interior and $L_2(D)$ and
$C^2(D)$ denote the set of square-integrable measurable functions and
the set of twice continuously differentiable functions over $D$,
respectively. Throughout the paper, $m(\cdot)$ denotes the Lebesgue
measure. If $\{E_k\}_{k = 1}^\infty$ is a sequence of subsets of
$\Omega$ such that $E_k \subseteq E_{k + 1}$ and $E = \bigcup_k E_k$,
then we write $E_k \uparrow E$ as $k \to \infty$. We say $E_k
\downarrow E$ as $k \to \infty$ if $E_k^c \uparrow E^c$ as $k \to
\infty$, where $E^c = \Omega \setminus E$ is the complement of
$E$. Given any closed and convex subset $\Sc \subseteq \Hc$ of a
Hilbert space, we denote by $\proj_\Sc$ the orthogonal projection
onto~$\Sc$.

We denote by $\Kc$ the set of strictly increasing continuous functions
$\alpha: [0, \infty) \to [0, \infty)$ such that $\alpha(0) = 0$. A
function $\alpha$ belongs to $\Kc_\infty$ if $\alpha \in \Kc$ and
$\lim_{r \to \infty} \alpha(r) = \infty$. We denote by $\Kc \Lc$ the
set of functions $\beta:[0, \infty) \times [0, \infty) \to [0,
\infty)$ such that, for each $s \in [0, \infty)$, $r \mapsto
\beta(r,s)$ is nondecreasing and continuous and $\beta(0, s) = 0$ and,
for each $r \in [0, \infty)$, $s \mapsto \beta(r,s)$ is monotonically
decreasing with $\beta(r, s) \to 0$ as $s \to \infty$.  A map $M: X
\to Y$ between two normed vector spaces is $\Kc$-Lipschitz if there
exists $\kappa \in \Kc_\infty$ such that $\|M(x_1) - M(x_2)\|_Y \le
\kappa(\|x_1 - x_2\|_X)$ for all $x_1, x_2 \in X$.

The (zero-mean) Laplace distribution with scale $b \in \realpos$ is a
continuous distribution with probability density function
\begin{align*}
  \Lc(x; b) = \frac{1}{2b} e^{-\frac{|x|}{b}}.
\end{align*}
It is clear that $\frac{\Lc(x; b)}{\Lc(y; b)} \le e^\frac{|x -
  y|}{b}$. We use $\eta \sim \Lap(b)$ to denote a random variable
$\eta$ with Laplace distribution. It is easy to see that if $\eta \sim
\Lap(b)$, $|\eta|$ has an exponential distribution with rate $\lambda
= \frac{1}{b}$. Similarly, we use the notation $\eta \sim \Nc(\mu,
\sigma^2)$ when $\eta$ is normally distributed with mean $\mu$ and
variance $\sigma^2$. The error function $\erf:\real \to \real$ is
defined as
\begin{align*}
  \erf(x) \triangleq \frac{1}{\sqrt \pi} \int_{-x}^x e^{-t^2} d t \ge
  1 - e^{-x^2}.
\end{align*}
Therefore, $\pr\{|\eta| \le r\} = \erf(r / \sqrt 2 \sigma)$ if $\eta
\sim \Nc(0, \sigma^2)$. Given any random variable $\eta$ and any
convex function $\phi$, Jensen's inequality states that
$\E[\phi(\eta)] \ge \phi(\E[\eta])$. The opposite inequality holds if
$\phi$ is concave.

\subsection{Hilbert Spaces and Orthonormal Bases}

We review some basic facts on Hilbert spaces and refer the reader
to~\cite{EK:89} for a comprehensive treatment.  A Hilbert space~$\Hc$
is a complete inner-product space. A set $\{e_k\}_{k \in I} \subset
\Hc$ is an orthonormal system if $\langle e_k, e_j \rangle = 0$ for $k
\neq j$ and $\langle e_k, e_k \rangle = \|e_k\|^2 = 1$ for all $k \in
I$. If, in addition, the set of linear combinations of 
$\{e_k\}_{k \in I}$ is dense in~$\Hc$, then $\{e_k\}_{k \in I}$ is an
orthonormal basis.  Here, $I$ might be uncountable: however, if $\Hc$
is separable (i.e., it has a countable dense subset), then any
orthonormal basis is countable. In this case, we have
\begin{align*}
  h = \sum_{k = 1}^\infty \langle h, e_k \rangle e_k,
\end{align*}
for any $h \in \Hc$. We define the coefficient sequence $\boldsymbol{\theta} \in
\real^\intpos$ by $\theta_k = \langle h, e_k \rangle$ for $k \in
\intpos$. Then, $\boldsymbol{\theta} \in \ell_2$ and, by Parseval's
identity, $ \|h\| = \|\boldsymbol{\theta}\|$.  For ease of notation,
we define $\Phi:\ell_2 \to \Hc$ to be the linear bijection that maps
the coefficient sequence $\boldsymbol{\theta}$ to $h$.  For an
arbitrary $D \subseteq \real^d$, $L_p(D)$ is a Hilbert space if and
only if $p = 2$, and the inner product is the integral of the product
of functions. Moreover, $L_2(D)$ is separable. In the remainder of the
paper, we assume $\{e_k\}_{k = 1}^\infty$ is an orthonormal basis for
$L_2(D)$ and $\Phi:\ell_2 \to L_2(D)$ is the corresponding linear
bijection between coefficient sequences and functions.

\subsection{Robust Stability of Discrete-Time
  Systems}\label{subsection:ISS}

We briefly present some definitions and results on robust stability of
discrete-time systems following~\cite{CC-ART:05}.  Given the vector
field $f: \real^n \times \real^m \to \real^n$, consider the 
system
\begin{align}\label{eq:discrete-dynamics}
  x(k + 1) = f(x(k), \eta(k)),
\end{align}
with state $\mathbf{x}: \intnonneg \to \real^n$ and input
$\boldsymbol{\eta}: \intnonneg \to \real^m$. Given an equilibrium
point $x^* \in \real^n$ of the unforced system, we say
that~\eqref{eq:discrete-dynamics} is
\begin{enumerate}
\item \emph{0-input locally asymptotically stable (0-LAS) relative to
    $x^*$} if, by setting $\boldsymbol{\eta} = 0$, there exists $\rho > 0$ and
  $\gamma \in \Kc \Lc$ such that, for every initial condition $x(0)
  \in B(x^*, \rho)$, we have for all $k \in \intnonneg$,
    \begin{align*}
      |x(k) - x^*| \le \gamma(|x(0) - x^*|, k).
    \end{align*}
  \item \emph{locally input-to-state stable (LISS) relative to $x^*$}
    if there exist $\rho > 0$, $\gamma \in \Kc \Lc$, and $\kappa \in \Kc$
    such that, for every initial condition $x(0) \in B(x^*, \rho)$ and
    every input satisfying $\|\boldsymbol{\eta}\|_\infty \le \rho$, we
    have
  \begin{align}\label{eq:LISS}
    \hspace{-0.14in} |x(k) - x^*| \le \max \{\gamma(|x(0) - x^*|, k),
    \kappa(|\boldsymbol{\eta}_{k - 1}|_\infty)\},
  \end{align}
  for all $k \in \intpos$. In this case, we refer to $\rho$ as the
  \emph{robust stability radius}
  of~\eqref{eq:discrete-dynamics} relative to $x^*$.
 \end{enumerate}

By definition, if the system~\eqref{eq:discrete-dynamics} is LISS,
then it is also 0-LAS. The converse is also true, cf.~\cite[Theorem
1]{CC-ART:05}. The following result is a local version of~\cite[Lemma
3.8]{ZPJ-YW:01} and states an important asymptotic behavior of LISS
systems.

\begin{proposition}\longthmtitle{Asymptotic gain of LISS
    systems}\label{prop:asym-gain}
  Assume system~\eqref{eq:discrete-dynamics} is LISS relative to~$x^*$
  with associated robust stability radius~$\rho$. If
  $x(0) \in B(x^*, \rho)$ and $\|\boldsymbol{\eta}\|_\infty \le
  \min\{\kappa^{-1}(\rho), \rho\}$ (where $\kappa^{-1}(\rho) = \infty$
  if $\rho$ is not in the range of $\kappa$), then
  \begin{align*}
    \limsup_{k \to \infty} |x(k) - x^*| \le \kappa(\limsup_{k \to
      \infty} |\eta(k)|).
  \end{align*}
  In particular, $x(k) \to x^*$ if $\eta(k) \to 0$ as $k \to \infty$.
\end{proposition}
\begin{proof}
  From~\eqref{eq:LISS}, we have
  \begin{align}\label{eq:auxx}
    |x(k) - x^*| &\le \max \{\gamma(|x(0) - x^*|, k),
    \kappa(\|\boldsymbol{\eta}\|_\infty)\} \notag
    \\
    &\le \max \{\gamma(\rho, k),
    \kappa(\|\boldsymbol{\eta}\|_\infty)\} ,
  \end{align}
  where we have used $x(0) \in B(x^*, \rho)$. Now, for each $k \in
  \intpos$, let $\boldsymbol{\eta}_{[k]} \in (\real^n)^\intpos$ be
  defined by $ \eta_{[k]}(\ell) = \eta(k + \ell)$ for all $\ell \in
  \intnonneg$.  If there exists $k_0$ such that
  $\|\boldsymbol{\eta}_{[k_0]}\|_\infty = 0$, then we need to show
  that $\lim_{k \to \infty} |x(k) - x^*| = 0$. Since $\gamma \in \Kc
  \Lc$, there exists $K \in \intnonneg$ such that $\gamma(\rho, k) \le
  \rho$ for all $k \ge K$, and since
  $\kappa(\|\boldsymbol{\eta}\|_\infty) \le \rho$ as well, it follows
  from~\eqref{eq:auxx} that $x(k) \in B(x^*, \rho)$ for all $k \ge
  K$. Let $\overline k = \max\{k_0, K\}$. Using~\eqref{eq:LISS}, we
  get
  \begin{align*}
    |x(k) - x^*| \le \gamma(|x(\overline k) - x^*|, k - \overline k),
    \quad \forall k > \overline k,
  \end{align*}
  and the result follows.  Assume then that no $k_0$ exists such that
  $\|\boldsymbol{\eta}_{[k_0]}\|_\infty = 0$.  Let $K_0 = 0$ and, for
  each $j \in \intpos$, let $K_j$ be such that $\gamma(\rho,k - K_{j -
    1}) \le \kappa(\|\boldsymbol{\eta}_{[K_{j-1}]}\|_\infty)$ for all
  $k \ge K_j$ (this sequence is well-defined because $\gamma \in
  \Kc\Lc$). Since $\kappa(\|\boldsymbol{\eta}_{[K_{j-1}]}\|_\infty)
  \le \kappa(\|\boldsymbol{\eta}\|_\infty) \le \rho$, \eqref{eq:LISS}
  holds if we set the ``initial'' state to $x(K_{j - 1})$ which
  implies that $|x(k) - x^*| \le
  \kappa(\|\boldsymbol{\eta}_{[K_{j-1}]}\|_\infty)$ for all $k \ge
  K_j$. Therefore,
  \begin{align*}
    \limsup_{k \to \infty} |x(k) - x^*| \le \kappa(
    \|\boldsymbol{\eta}_{[K_{j}]}\|_\infty), \quad \forall j \in
    \intnonneg.
  \end{align*}
  The result follows by taking limit of both sides as $j \to \infty$.
\end{proof}

\section{Problem Statement}\label{sec:prob-state}

Consider a group of $n$ agents whose communication topology is
described by a digraph $\Gc$. Each agent $i \in \until{n}$ has a local
objective function $f_i: D \to \real$, where $D \subset \real^d$ is
convex and compact and has nonempty interior. We assume that each
$f_i, i \in \until{n}$ is convex and twice continuously
differentiable, and use the shorthand notation $F = \{f_i\}_{i =
  1}^n$.  Consider the following convex optimization problem
\begin{align*}
  \minimize_{x \in D} \quad &f(x) \triangleq \sum_{i = 1}^n f_i(x)
  \\
  \notag \text{subject to } \ \, &G(x) \le 0,
  \\
  \notag \qquad \qquad &A x = b,
\end{align*}
where the component functions of $G:D \to \real^m$ are convex, $A \in
\real^{s \times d}$, and $b \in \real^s$. Denote by $X \subseteq D$
the feasibility set. The optimization problem can be
equivalently written as,
\begin{align}\label{eq:f-X}
  \minimize_{x \in X} \ f(x).
\end{align}
We assume that $X$ is a global piece of information known to all
agents.

The group objective is to solve the convex optimization
problem~\eqref{eq:f-X} in a distributed and private way. By
distributed, we mean that each agent can only interact with its
neighbors in~$\Gc$. For privacy, we consider the case where the
function $f_i$ (or some of its attributes) constitute the local and
sensitive information known to agent $i \in \until{n}$ that has to be
kept confidential. Each agent assumes (the worst-case where) the
adversary has access to all the ``external'' information (including
all the network communications and all other objective
functions). This setting is sometimes called local (differential)
privacy in the literature, see e.g.,~\cite{JCD-MIJ-MJW:13}.
In order to define privacy, we first introduce the notion
of adjacency. Given any normed vector space $(\Vc, \|\cdot\|_\Vc)$
with $\Vc \subseteq L_2(D)$, two sets of functions $F, F' \subset
L_2(D)$ are $\Vc$-adjacent if there exists $i_0 \in \until{n}$ such
that
\begin{align*}
  f_i = f_i', \; i \neq i_0 \quad \text{and} \quad
  f_{i_0} - f_{i_0}' \in \Vc.
\end{align*}
The set $\Vc$ is a design choice that we specify later
  in Section~\ref{subsec:func-perturb}. Moreover, this definition can
be readily extended to the case where $\Vc$ is any subset of another
normed vector space $\Wc \subseteq L_2(D)$. With this generalization,
the conventional bounded-difference notion of adjacency becomes a
special case of the definition above, where $\Vc$ is a closed ball
around the origin. We provide next a general definition of
differential privacy for a map.

\begin{definition}\longthmtitle{Differential
    Privacy}\label{def:func-diff-privacy}
  Let $(\Omega, \Sigma, \pr)$ be a probability space and consider a
  random map
  \begin{align*}
    \Mc: L_2(D)^n \times \Omega \to \Xc
  \end{align*}
  from the function space $L_2(D)^n$ to an arbitrary set $\Xc$. Given
  $\epsilon \in \realpos^n$, the map $\Mc$ is
  $\epsilon$-differentially private if, for any two $\Vc$-adjacent
  sets of functions $F$ and $F'$ that (at most) differ in
  their $i_0$'th element and any set $\Oc \subseteq \Xc$, one has
  \begin{align}\label{eq:diff-privy}
    \pr \setdef{\omega \in &\Omega}{\Mc(F', \omega) \in \Oc}
    \\
    \notag &\le e^{\epsilon_{i_0} \|f_{i_0} -
      f_{i_0}'\|_\Vc} \pr \setdef{\omega \in \Omega}{\Mc(F,
    \omega) \in \Oc}. \eqoprocend
  \end{align}
\end{definition}

Essentially, this notion requires the statistics of the output
of~$\Mc$ to change only (relatively) slightly if the objective
function of one agent changes (and the change is in $\Vc$), making it
hard for an adversary who observes the output of~$\Mc$ to determine
the change. In the case of an iterative asymptotic distributed
optimization algorithm, $\Mc$ represents the action (observed by the
adversary) of the algorithm on the set of local functions $F$. In
other words, $\Mc$ is the map (parameterized by the initial network
condition) that assigns to $F$ the whole sequence of messages
transmitted over the network. In this case,~\eqref{eq:diff-privy} has
to hold for all allowable initial conditions. We are ready to formally
state the network objective.

\begin{problem}\longthmtitle{Differentially private distributed
    optimization}\label{problem}
  Design a distributed and differentially private optimization
  algorithm whose guarantee on accuracy improves as the level of
  privacy decreases, leading to the exact optimizer of the aggregate
  objective function in the absence of privacy.  \oprocend
\end{problem}

The reason for the requirement of recovering the exact optimizer in
the absence of privacy in Problem~\ref{problem} is the following. It
is well-known in the literature of differential privacy that there
always exists a cost for an algorithm to be differentially private,
i.e., the algorithm inevitably suffers a performance loss that
increases as the level of privacy increases. This phenomenon is a
result of the noise added in the map~$\Mc$, whose variance increases
as $\epsilon$ decreases. With the above requirement on the noise-free
behavior of the algorithm, we aim to make sure that the cause of this
performance loss is \emph{only} due to the added noise and not to any
other factor.

\begin{example}\longthmtitle{Linear Classification with
    Logistic Loss Function}\label{example} 
  {\rm We introduce here a supervised classification problem that will
    serve to illustrate the discussion along the paper. Consider a
    database of training records composed by the labeled samples
    $\{(a_i, b_i)\}_{i = 1}^N$, where each $a_i \in \real^d$
    (containing the features of a corresponding object) may belong to
    one of two possible classes and $b_i \in \{-1, 1\}$ determines to
    which class it belongs. The goal is to train a classifier with the
    samples so that it can automatically classify future unlabeled
    samples. For simplicity, we let $d = 2$ and assume $a_i \in [0,
    1]^2$ and $b_i \in \{-1, 1\}$ are independently and uniformly
    randomly selected.  The aim is to find the best hyperplane $x^T a$
    that can separate the two classes. The parameters $x$ defining the
    hyperplane can be found by solving the convex problem,
    \begin{align}\label{eq:classifier}
      x^* = 
      \argmin_{x \in X} \sum_{i = 1}^N \Big(\ell(x; a_i, b_i) +
      \frac{\lambda}{2} |x|^2\Big),
    \end{align}
    where $\ell:\real^d \times \real^d \times \real \to \realpos$ is
    the loss function and $( \lambda/2 ) |x|^2$ is the regularizing
    term.  Since the objective function is strongly convex, we choose
    $X$ large enough so that $x^*$ is the same as the unique
    unconstrained minimizer. Popular choices of $\ell$ are the
    logistic loss $\ell(x; a_i, b_i) = \ln(1 + e^{-b_i a_i^T x})$ and
    the hinge loss $\ell(x; a_i, b_i) = \max\{0, 1 - b_i a_i^T
    x\}$. We focus on the logistic loss here due to its smoothness.

    Consider a group of $n$ agents, each one owning a portion $N_d = N
    / n$ of the training samples, who seek to collectively
    solve~\eqref{eq:classifier} in a distributed fashion, i.e., only
    by communicating with their neighbors (without a central
    aggregator). Various iterative algorithms have been proposed in
    the literature,
    cf.~\cite{AN-AO-PAP:10,BJ-MR-MJ:09,MZ-SM:12,BG-JC:14-tac,JCD-AA-MJW:12},
    to address this problem.  As an
    example,~\cite{AN-AO-PAP:10} proposes that each agent $i \in
    \until{n}$ starts with an initial estimate $x_i(0)$ of $x^*$ and,
    at each iteration $k$, update its estimate as
    \begin{subequations}
      \begin{align}
        x_i(k + 1) &= \proj_{X} (z_i(k) - \alpha_k \nabla
        f_i(z_i(k))), \label{eq:nedic-a}
        \\
        \label{eq:nedic-b} z_i(k) &= \sum_{j = 1}^n a_{i j} x_j(k),
      \end{align}
    \end{subequations}
    where $\{a_{i j}\}_{j = 1}^n$ are the edge weights of the
    communication graph at node $i$ and $\alpha_k$ is the
    stepsize. From~\eqref{eq:nedic-b}, one can see that agents only
    need to share their estimates with their neighbors to run the
    algorithm. Under reasonable connectivity assumptions, one can
    show~\cite{AN-AO-PAP:10} that $x_i(k)$ converges to $x^*$
    asymptotically if the sequence of stepsizes is square-summable
    ($\sum_k \alpha_k^2 < \infty$) but not summable ($\sum_k \alpha_k
    = \infty$).  In this paper, we are interested in endowing
    distributed coordination algorithms such as this with privacy
    guarantees so that their execution does not reveal information
    about the local objective functions to the adversary.} \oprocend
\end{example}

\section{Rationale for Design Strategy}\label{sec:rationale}

In this section, we discuss two algorithm design strategies to solve
Problem~\ref{problem} based on the perturbation of either inter-agent
messages or the local objective functions. We point out an important
limitation of the former, and this provides justification for the
ensuing design of our objective-perturbing algorithm based on
functional differential privacy.

\subsection{Limitations of Message-Perturbing
  Strategies}\label{subsec:mess-perturb}

We use the term \emph{message-perturbing strategy} to refer to the result of
modifying any of the distributed optimization algorithms available in
the literature by adding (Gaussian or Laplace) noise to the messages
agents send to either neighbors or a central aggregator in order to
preserve privacy.  A generic message-perturbing
distributed algorithm takes the form
\begin{equation}
  \begin{aligned}
    x(k + 1) &= a_\Ic(x(k), \xi(k)) ,
    \\
    \xi(k) &= x(k) + \eta(k) ,
  \end{aligned}
  \label{eq:nonlin-dyn}
\end{equation}
  where $\boldsymbol{\xi}$, $\boldsymbol{\eta}: \intnonneg \to
  \real^n$ are the sequences of messages and perturbations,
  respectively, and $a_\Ic:\real^n \times \real^n \to \real^n$ depends
  on the agents' sensitive information set~$\Ic$ with associated
  optimizer $x_\Ic^*$. This formulation is quite general and can also
encode algorithmic solutions for optimization problems other than the
one in Section~\ref{sec:prob-state}, such as the ones
  studied in~\cite{SH-UT-GJP:16,MTH-ME:15}. In the
problem of interest here,  $\Ic = F = \{f_i\}_{i = 1}^n$.

The following result provides conditions on the noise variance that
ensure that the noise vanishes asymptotically almost surely and
remains bounded with nonzero probability. 

\begin{lemma}\longthmtitle{Convergence and boundedness of Laplace and
    normal random sequences with decaying
    variance}\label{lem:conv-bound}
  Let $\boldsymbol{\eta}$ be a sequence of independent random
  variables defined over the sample space $\Omega = \real^\intpos$,
  with $\eta(k) \sim \Lap(b(k))$ or $\eta(k) \sim \Nc(0, b(k))$ for
  all $k \in \intpos$. Given $r > 0$, consider the events
  \begin{align*}
    E &= \setdef{\boldsymbol{\eta} \in \Omega}{\lim_{k \to \infty}
      \eta(k) = 0},
    \\
    F_r &= \setdef{\boldsymbol{\eta} \in \Omega}{\forall k \in \intpos
      \quad |\eta(k)| \le r}.
  \end{align*}
  If $b(k)$ is $O(\frac{1}{k^p})$ for some $p > 0$, then
  $\pr(E) = 1$ and $\pr(F_r) = \pr(F_r \cap E) > 0$ for all $r > 0$.
\end{lemma}
\begin{proof}
  First, consider the case where $\eta(k) \sim \Lap(b(k))$. By the
  independence of the random variables and the fact that $|\eta(k)|$ is
  exponentially distributed with rate $\frac{1}{b(k)}$,
  \begin{align*}
    \pr(F_r) = \prod_{k = 1}^\infty \left(1 - e^{-\frac{r}{b(k)}}
    \right).
  \end{align*}
  By assumption, $b(k) \le \frac{c}{k^p}$ for all $k \in \intpos$ and
  some $p, c > 0$. Thus, given that the series $\sum_{k = 1}^\infty
  e^{-\frac{r}{c} k^p}$ converges~\cite[\S 1.14]{HJ-BSJ:99},
  \begin{align*}
    \pr(F_r) \ge \prod_{k = 1}^\infty \left(1 - e^{-\frac{r}{c} k^p}
    \right) > 0.
  \end{align*}
  Next, let $E_{\ell, K} = \setdef{\boldsymbol{\eta} \in
    \Omega}{\forall k \ge K \quad |\eta(k)| < \upsilon_\ell}$ where
  $\{\upsilon_\ell\}_{\ell = 1}^\infty$ is a monotonically decreasing
  sequence that converges to zero as $\ell \to \infty$ (e.g.,
  $\upsilon_\ell = \frac{1}{\ell}$). Note that
  \begin{align*}
    E = \bigcap_{\ell = 1}^\infty \bigcup_{K = 1}^\infty E_{\ell, K}.
  \end{align*}
  Since $E_{\ell, K} \uparrow \bigcup_{K = 1}^\infty E_{\ell, K}$ for all
  $\ell \in \intpos$ as $K \to \infty$, and $\bigcup_{K = 1}^\infty
  E_{\ell, K} \downarrow E$ as $\ell \to \infty$, we have
  \begin{align*}
    \pr(E) &= \lim_{\ell \to \infty} \lim_{K \to \infty} \pr(E_{\ell,
      K}) = \lim_{\ell \to \infty} \lim_{K \to \infty} \prod_{k =
      K}^\infty \left(1 - e^{-\frac{\upsilon_\ell}{b(k)}} \right)
    \\
    &\ge \lim_{\ell \to \infty} \lim_{K \to \infty} \prod_{k =
      K}^\infty \left(1 - e^{-\frac{\upsilon_\ell}{c} k^p} \right) = 1 .
  \end{align*}
  Then, $\pr(F_r \cap E) = \pr(F_r) - \pr(F_r \cap E^c) = \pr(F_r) >
  0$.  For the case of normal distribution of random variables,
  \begin{align*}
    \pr\{|\eta(k)| \le r\} = \erf \bigg(\frac{r}{\sqrt{2 b(k)}}\bigg)
    \ge 1 - e^{-\frac{r^2}{2 b(k)}},
  \end{align*}
  and the results follows from the arguments above.
\end{proof}

Note that Lemma~\ref{lem:conv-bound} also ensures that the probability
that the noise simultaneously converges to zero and remains bounded is
nonzero. One might expect that Lemma~\ref{lem:conv-bound} would hold
if $b(k) \to 0$ at any rate. However, this is not true. For instance,
if $b(k) = \frac{1}{\log k}$, one can show that the probability that
$\eta(k)$ eventually remains bounded is zero for any bound $r \le 1$,
so the probability that $\eta(k) \to 0$ is zero as well.

The following result shows that a message-perturbing algorithm of the
form~\eqref{eq:nonlin-dyn} cannot achieve differential privacy if the
underlying (noise-free) dynamics are asymptotically
stable. For convenience, we employ the short-hand notation $\tilde
a_\Ic(x(k), \eta(k)) = a_\Ic(x(k), x(k) + \eta(k))$ to refer
to~\eqref{eq:nonlin-dyn}.

\begin{proposition}\longthmtitle{Impossibility result for 0-LAS
    message-perturbing algorithms}\label{prop:impos-result}
  Consider \emph{any} algorithm of the form~\eqref{eq:nonlin-dyn} with
  either $\eta_i(k) \sim \Lap(b_i(k))$ or $\eta_i(k) \sim \Nc(0,
  b_i(k))$. If $\tilde a_\Ic$ is 0-LAS relative to $x^*_\Ic$ for two
  information sets $\Ic$ and $\Ic'$ with different optimizers
  $x_{\Ic}^* \neq x_{\Ic'}^*$ and associated robust stability radii
  $\rho$ and $\rho'$, respectively, $b_i(k)$ is $O(\frac{1}{k^p})$ for
  all $i \in \until{n}$ and some $p > 0$, and at least one of the
  following holds,
  \begin{enumerate}
  \item $x_{\Ic}^*$ is not an equilibrium point of $x(k + 1) =
    \tilde a_{\Ic'}(x(k), 0)$ and $\tilde a_{\Ic'}$ is
    continuous,
  \item $x_{\Ic}^*$ belongs to the interior of
    $B(x_{\Ic'}^*, \rho')$,
  \end{enumerate}
  then, the algorithm cannot preserve the $\epsilon$-differentially
  privacy of the information set $\Ic$ for any $\epsilon>0$.
\end{proposition}
\begin{proof}
  Our proof strategy consists of establishing that, if
    the initial state is close to the equilibrium of the system for
    one information set, the state trajectory converges to that
    equilibrium with positive probability but to the equilibrium of
    the system with the other information set with probability
    zero. We then use this fact to rule out differential privacy.
  For any fixed initial state $x_0$, if either of $\boldsymbol{\xi}$
  or $\boldsymbol{\eta}$ is known, the other one can be uniquely
  determined from~\eqref{eq:nonlin-dyn}. Therefore, the mapping
  $\Xi_{\Ic, x_0}:(\real^n)^\intpos \to (\real^n)^\intpos$ such that
  \begin{align*}
    \Xi_{\Ic, x_0}(\boldsymbol{\eta}) = \boldsymbol{\xi}
  \end{align*}
  is well-defined and bijective. Let $\kappa, \kappa' \in \Kc$ be as
  in~\eqref{eq:LISS} corresponding to $\tilde a_{\Ic}$ and $\tilde
  a_{\Ic'}$, respectively. Consider as initial condition $x_0 =
  x_{\Ic}^*$ and define
  \begin{align*}
    R = & \big\{ \boldsymbol{\eta} \in \Omega \; | \; \forall i
    \in \until{n}, \; \lim_{k \to \infty} \eta_i(k) = 0
    \\
    & \quad \text{and} \; |\eta_i(k)| \le
    \min\big\{\kappa^{-1}(\rho), \rho\big\} , \;
    \forall k \in \intpos \big\} .
  \end{align*}
  By Lemma~\ref{lem:conv-bound}, we have $\pr(R) > 0$.  By
  Proposition~\ref{prop:asym-gain}, since $|x_0 - x_{\Ic}^*| = 0 \le
  \rho$ and $\|\boldsymbol{\eta}\|_\infty \le
  \min\big\{\kappa^{-1}(\rho), \rho\big\}$ for all $\boldsymbol{\eta}
  \in R$, the sequence $\Xi_{\Ic,x_0}(\boldsymbol{\eta})$ converges to
  $x_{\Ic}^*$. Let $\Oc = \Xi_{\Ic, x_0}(R)$ and $R' = \Xi_{\Ic',
    x_0}^{-1} (\Oc)$ (where we are using the forward
    and inverse images of sets, respectively). Next, we show that no
  $\boldsymbol{\eta}' \in R'$ converges to $0$ under either hypothesis
  (i) or (ii) of the statement.  Under (i),
  there exists a neighborhood of $(x_{\Ic}^*, 0) \in \real^{2n}$ in
  which the infimum of the absolute value of at least one of the
  components of $\tilde a_{\Ic'}(x, \eta)$ is positive, so whenever
  $(x, \eta)$ enters this neighborhood, it exits it in finite
  time. Therefore, given that any $\mathbf{x} \in \Oc$ converges to
  $x_{\Ic}^*$, no $\boldsymbol{\eta}' \in R'$ can converge to zero.
  Under (ii),
  there exists a neighborhood of $x_{\Ic}^*$ included in
  $B(x_{\Ic'}^*, \rho')$. Since $\Xi_{\Ic',
    x_0}(\boldsymbol{\eta}') \to x_{\Ic}^*$, there exists
  $K \in \intpos$ such that $\Xi_{\Ic',
    x_0}(\boldsymbol{\eta}')(k)$ belongs to $B(x_{\Ic'}^*,
  \rho')$ for all $k \ge K$. Therefore, if $|\eta'(k)| \le
  \min\big\{(\kappa')^{-1}(\rho'), \rho'\big\}$
  indefinitely after any point of time, $\Xi_{\Ic',
    x_0}(\boldsymbol{\eta}') \to x_{\Ic'}^*$ by
  Proposition~\ref{prop:asym-gain} which is a contradiction, so
  $\boldsymbol{\eta}'$ cannot converge to zero.  In both cases,
  by Lemma~\ref{lem:conv-bound},
    $\pr(R') = 0$, 
  which, together with $\pr(R) > 0$ and the definition of
  $\epsilon$-differential privacy, cf.~\eqref{eq:diff-privy}, implies
  the result.
\end{proof}

Note that the hypotheses of Proposition~\ref{prop:impos-result} are
mild and easily satisfied in most cases. In particular, the result
holds if the dynamics are continuous and globally
asymptotically stable relative to $x^*_\Ic$ for two information sets.
The main take-away message of this result is that a globally
asymptotically stable distributed optimization algorithm cannot be
made differentially private by perturbing the inter-agent messages
with asymptotically vanishing noise. This observation is at the core
of the design choices made in the literature regarding the use of
stepsizes with finite sum to make the zero-input dynamics not
asymptotically stable, thereby causing a steady-state error in
accuracy which is present independently of the amount of noise
injected for privacy.  For instance, the algorithmic solution proposed
in~\cite{ZH-SM-NV:15} replaces~\eqref{eq:nedic-b} by
  $z_i(k) = \sum_{j = 1}^n a_{i j} \xi_j(k)$, where $\xi_j(k) = x_j(k)
  + \eta_j(k)$ is the perturbed message received from agent $j$, and
  chooses a finite-sum sequence of stepsizes $\{\alpha_k\}$ in the
  computation~\eqref{eq:nedic-a}, leading to a dynamical system which
is not 0-GAS, see Figure~\ref{fig:huang}.
\begin{figure}[htb]
  \centering
  \includegraphics[width = \linewidth]{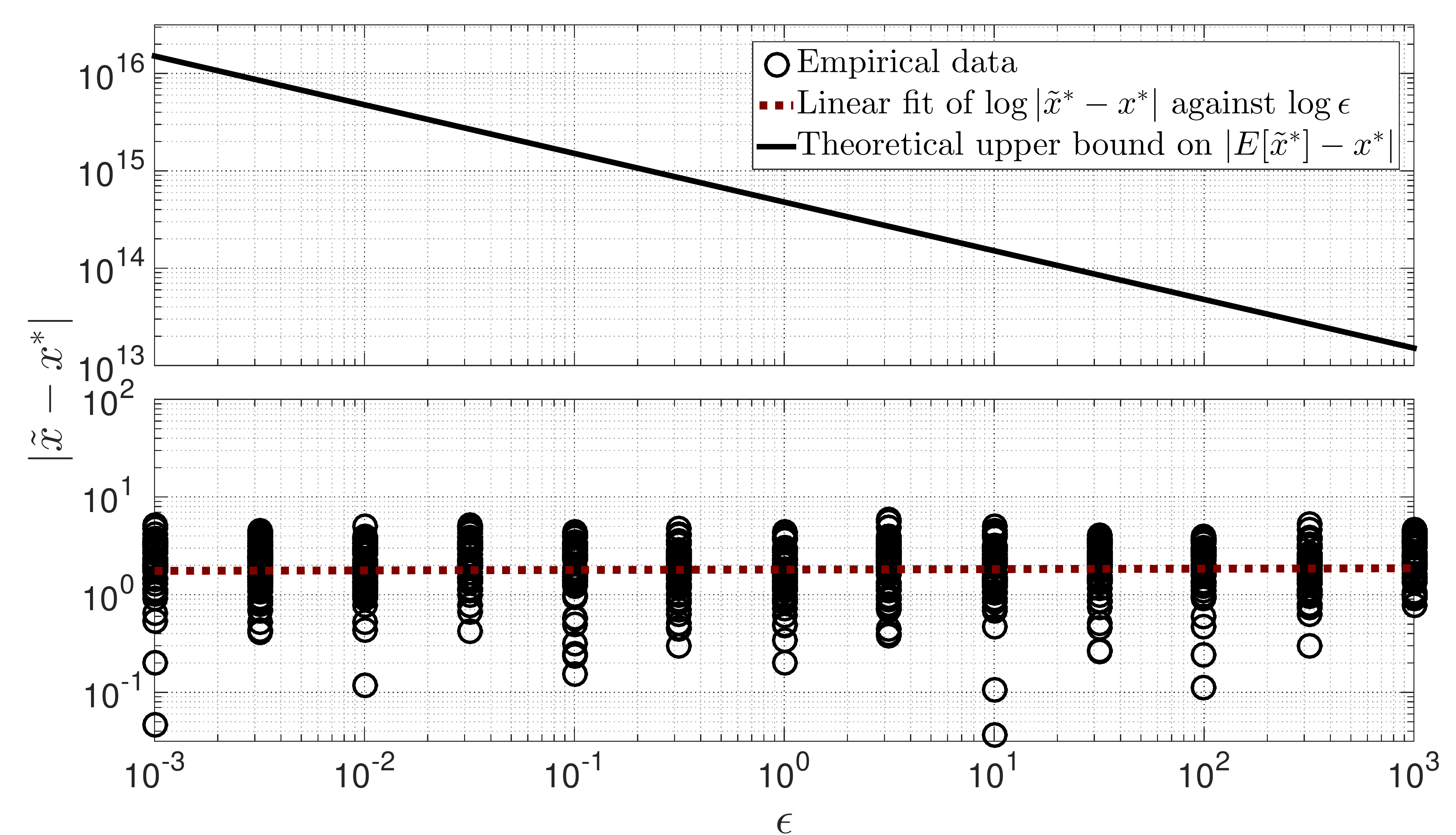}
  \caption{Privacy-accuracy trade-off for the algorithm proposed
    in~\cite{ZH-SM-NV:15} applied to Example~\ref{example} with $D = X
    = [-5, 5]^2$, $n = 10$, $N_d = 100$, and $\lambda = 0.01$. With
    that paper's notation, we set $q = 0.1$, $p = 0.11$, $c =
    0.5$. The stepsize $\alpha_k = c q^{k - 1}$ has finite sum.  The
    circles, dotted line, and solid line illustrate simulation results
    for 50 executions, their best linear fit in logarithmic scale, and
    the upper bound on accuracy provided in~\cite{ZH-SM-NV:15},
    respectively. We have broken the vertical axis to better display
    the scale of the algorithm output.}\label{fig:huang}
\end{figure}
Similar observations can be made in the scenario considered
in~\cite{SH-UT-GJP:16}, where the agents' local constraints are
the sensitive information (instead of the objective function). This
algorithmic solution uses a constant-variance noise, which would make
the dynamics unstable if executed over an infinite time horizon. This
problem is circumvented by having the algorithm terminate after a
finite number of steps, and optimizing this number offline as a
function of the desired level of privacy~$\epsilon$.

\subsection{Algorithm Design via Objective
  Perturbation}\label{sec:algorithm-design-plan}

To overcome the limitations of message-perturbing strategies, here we
outline an alternative design strategy to solve Problem~\ref{problem}
based on the perturbation of the agents' objective functions.  The
basic idea is to have agents independently perturb their objective
functions in a differentially private way and then have them
participate in a distributed optimization algorithm with the perturbed
objective functions instead of their original ones. In
  the context of Example~\ref{example}, this would correspond to
  leave~\eqref{eq:nedic-b} and the sequence of stepsizes unchanged,
  and instead use perturbed functions in the
  computation~\eqref{eq:nedic-a}. The latter in turn automatically
  adds noise to the estimates shared with neighbors. The following
result, which is a special case of \cite[Theorem 1]{JLN-GJP:14},
ensures that the combination with the distributed optimization
algorithm does not affect the differential privacy at the functional
level.

\begin{proposition}\longthmtitle{Resilience to
    post-processing}\label{prop:post-proc}
  Let $\Mc: L_2(D)^n \times \Omega \to L_2(D)^n$ be
  $\epsilon$-differentially private
    (cf. Definition~\ref{def:func-diff-privacy}) and $\Fc: L_2(D)^n
  \to \Xc$, where $(\Xc, \Sigma_\Xc)$ is an arbitrary measurable
  space. Then, $\Fc \circ \Mc: L_2(D)^n \times \Omega \to \Xc$ is
  $\epsilon$-differentially private.
\end{proposition}
\begin{proof}
  First, note that although a slightly different definition of
  differential privacy is used in~\cite{JLN-GJP:14}, the exact same
  proof of~\cite[Theorem 1]{JLN-GJP:14} works with
  Definition~\ref{def:func-diff-privacy}. Consider the
  $\sigma$-algebra $\Pc(L_2(D)^n)$ on $L_2(D)^n$ where $\Pc$ denotes
  the power set. With the notation of~\cite[Theorem 1]{JLN-GJP:14},
  $M_2 = \Fc \circ \Mc$ is a deterministic function of the output of
  $M_1 = \Mc$. Then, it is easy to verify that, for any $S \in
  \Sigma_\Xc$,
  \begin{align*}
    \pr(M_2(F) \in S \;| \; M_1(F)) = \chi_S(\Fc(M_1(F))),
  \end{align*}
  (with $\chi_\cdot$ being the indicator function) is measurable as a
  function of $M_1(F)$ (because $\Fc$ and $S$ are trivially
  measurable) and defines a probability measure on $(\Xc, \Sigma_\Xc)$
  (associated to a singleton), so it is a probability
  kernel. Hence, the conditions of~\cite[Theorem
    1]{JLN-GJP:14} are satisfied and $\Fc \circ \Mc$ is
    $\epsilon$-differentially private.
\end{proof}

Our design strategy based on the perturbation of individual
objective functions requires solving the following challenges:
\begin{enumerate}
\item establishing a differentially private procedure to perturb the
  individual objective functions;
\item ensuring that the resulting perturbed functions enjoy the
  smoothness and regularity properties required by distributed
  optimization algorithms to converge;
\item with (i) and (ii) in place, characterizing the accuracy of the
  resulting differentially private, distributed coordination
  algorithm.
\end{enumerate}

Section~\ref{sec:func-diff-privacy} addresses~(i) and
Section~\ref{sec:diff-private-dist-opt} deals with~(ii) and~(iii).

\section{Functional Differential Privacy}\label{sec:func-diff-privacy}

We explore here the concept of functional differential privacy to
address the challenge~(i) laid out in
Section~\ref{sec:algorithm-design-plan}.  The generality of this
notion makes it amenable for problems where the sensitive information
is a function or some of its attributes (e.g., sample points,
optimizers, derivatives and integrals). For simplicity of exposition
and without loss of generality, we limit our discussion in this
section to the privacy of a single function.

\subsection{Functional Perturbation via Laplace Noise}

Let $f \in L_2(D)$ be a function whose differential privacy has to be
preserved. With the notation of Section~\ref{sec:prelims}, we
decompose $f$ into its coefficients $\Phi^{-1}(f)$ and perturb this
sequence by adding noise to all of its elements. Specifically, we set
\begin{align}\label{eq:f-hat}
  \Mc(f, \boldsymbol{\eta}) = \Phi \left(\Phi^{-1}(f) +
    \boldsymbol{\eta} \right) = f + \Phi(\boldsymbol{\eta}),
\end{align}
where
\begin{align}\label{eq:eta}
  \eta_k \sim \Lap(b_k),
\end{align}
for all $k \in \intpos$. Clearly, for $\boldsymbol{\eta}$ to belong to
$\ell_2$ and for the series in $\Phi(\boldsymbol{\eta})$ to converge, the
scales $\lbrace b_k \rbrace_{k = 1}^\infty$ cannot be arbitrary.  The
next result addresses this issue.

\begin{lemma}\longthmtitle{Sufficient condition for boundedness of
    perturbed functions}\label{lem:l_2-condition}
  If there exists $K \in \intpos$ such that, for some $ p >
  \frac{1}{2}$ and $s > 1$,
  \begin{align}\label{eq:noise-decay}
    b_k \le \frac{1}{k^p \log k^s}, \quad \forall k \ge K ,
  \end{align}
  then $\boldsymbol{\eta}$ defined by~\eqref{eq:eta} belongs to $
  \ell_2$ with probability one. In particular, if for some $p >
  \frac{1}{2}$ and $\gamma > 0$,
  \begin{align}\label{eq:noise-decay-simple}
    b_k \le \frac{\gamma}{k^p}, \quad \forall k \in \intpos ,
  \end{align}
  then $\boldsymbol{\eta}$ defined by~\eqref{eq:eta} belongs to
  $\ell_2$ with probability one.
\end{lemma}
\begin{proof}
  Equation~\eqref{eq:noise-decay} can be equivalently written as $
  e^{-\frac{1}{k^p b_k}} \le \frac{1}{k^s}$, for $k \ge K$. In
  particular, this implies that $\sum_{k=1}^\infty e^{-\frac{1}{k^p
      b_k}}$ is convergent. Therefore~\cite[\S 1.14]{HJ-BSJ:99},
  $\prod_{k = 1}^\infty \big(1 - e^{-\frac{1}{k^p b_k}} \big)$
  converges (i.e., the limit exists and is nonzero), so
  \begin{align*}
    1 &= \lim_{K \to \infty} \prod_{k = K}^\infty \left(1 -
      e^{-\frac{1}{k^p b_k}}\right) = \lim_{K \to \infty} \pr (E_K),
  \end{align*}
  where $E_K = \setdefb{\boldsymbol{\eta} \in \real^\intpos}{\forall k
    \ge K, \; |\eta_k| \le \frac{1}{k^p}}$ and we have used the fact
  that $|\eta_k|$ is exponentially distributed with rate
  $\frac{1}{b_k}$. Since $E_K \uparrow \bigcup_{K = 1}^\infty E_K$ as
  $K \to \infty$, we have
  \begin{align*}
    1 \!&=\! \pr \Big(\!\bigcup_{K = 1}^\infty E_K \!\Big)
    \!=\! \pr \setdefB{\boldsymbol{\eta} \in \real^\intpos\!}{\!\exists K \in
      \intpos \; \text{s.t.} \; \forall k \ge K\!\!:\! |\eta_k| \!\le
      \!\frac{1}{k^p}\!}
    \\
    & \le \pr \lbrace \boldsymbol{\eta} \in \ell_2 \rbrace,
  \end{align*}
  as stated. If equation~\eqref{eq:noise-decay-simple} holds, we
  define $\bar p = \frac{1}{2}(p + \frac{1}{2})$ and equivalently
  write~\eqref{eq:noise-decay-simple} as
  \begin{align*}
    b_k \le \frac{1}{k^{\bar p}} \frac{\gamma}{k^{p - \bar p}}, \quad
    \forall k \in \intpos .
  \end{align*}
  Since $p - \bar p > 0$, for any $s > 1$ there exists $K \in \intpos$
  such that $k^{p - \bar p} \ge \gamma \log k^s$ for all $k \ge K$,
  and the result follows.
\end{proof}

Having established conditions on the noise variance under which the
map~\eqref{eq:f-hat} is well defined, we next turn our attention to
establish its differential privacy.

\subsection{Differential Privacy of Functional
  Perturbation}\label{subsec:func-perturb}

Here, we establish the differential privacy of the
map~\eqref{eq:f-hat}. In order to do so, we first specify our choice
of adjacency space $\Vc$. Given $q > 1$, consider the weight sequence
$\{k^q\}_{k = 1}^\infty$ and define the adjacency vector space to be
the image of the resulting weighted $\ell_2$ space under $\Phi$, i.e.,
\begin{align}\label{eq:Vq}
  \Vc_q = \Phi \big( \setdefb{\boldsymbol{\delta} \in
    \real^\intpos}{\sum_{k = 1}^\infty (k^q \delta_k)^2 < \infty}
  \big).
\end{align}
It is not difficult to see that $\Vc_q$ is a vector space. Moreover,
\begin{align*}
  \| f \|_{\Vc_q} \triangleq \Big( \sum_{k = 1}^\infty (k^q
  \delta_k)^2 \Big)^\frac{1}{2} , \quad \text{with }
  \boldsymbol{\delta} = \Phi^{-1}( f),
\end{align*}
is a norm on~$\Vc_q$. The next result establishes the differential
privacy of~\eqref{eq:f-hat} for a properly chosen noise
scale sequence $\mathbf{b}$.

\begin{theorem}\longthmtitle{Differential privacy of functional
    perturbation}\label{thm:diff-privacy}
  Given $q>1$, $\gamma>0$ and $p \in \left (\frac{1}{2}, q -
    \frac{1}{2} \right)$, let
  \begin{align}\label{eq:b}
    b_k = \frac{\gamma}{k^p}, \quad k \in \intpos.
  \end{align}
  Then, the map~\eqref{eq:f-hat} is $\epsilon$-differentially private
  with
  \begin{align}\label{eq:eps}
    \epsilon = \frac{1}{\gamma} \sqrt{\zeta(2 (q - p))},
  \end{align}
  where $\zeta$ is the Riemann zeta function.
\end{theorem}
\begin{proof}
 Note that the map $\Mc$ defined by~\eqref{eq:f-hat} is
    well defined because~\eqref{eq:b} ensures, by
    Lemma~\ref{lem:l_2-condition}, that $\boldsymbol{\eta}$ belongs to
    $ \ell_2$ almost surely.  Our proof consists of showing that $\Mc$
    satisfies the definition of differential privacy,
    cf. Definition~\ref{def:func-diff-privacy}. To this effect,
  consider two functions $f$ and $f'$, with $ f - f' \in \Vc_q$, and
  an arbitrary set $\Oc \subseteq L_2(D)$. Let $\Phi_K^{-1}:L_2(D) \to
  \real^K$ be the map that returns the first $K$ coefficients of
  $\Phi^{-1}(\cdot)$ and
   \begin{align*}
     \Lc^K(\boldsymbol\eta_K; \mathbf b_K) \triangleq \prod_{k =
       1}^K \Lc(\eta_k; b_k).
   \end{align*}
   We have
   \begin{align*}
    \pr \lbrace f + \Phi(\boldsymbol\eta) \in \Oc
    \rbrace &= \pr \lbrace \boldsymbol\eta \in \Phi^{-1} (\Oc
    - f) \rbrace
    \\
    &= \lim_{K \to \infty} \int_{\Phi^{-1}_K (\Oc - f)}
    \Lc^K(\boldsymbol\eta_K; \mathbf b_K) d
    \boldsymbol\eta_K,
  \end{align*}
  where $\Phi^{-1}_K (\Oc - f)$ denotes the inverse image of the set
  $\Oc - f = \setdef{g \in L_2(D)}{g + f \in \Oc}$ and the second
  equality follows from the continuity of
  probability~\cite[Theorem~1.1.1.iv]{RD:10} (since $\Phi^{-1}_K (\Oc
  - f) \times \real^\intpos \downarrow \Phi^{-1} (\Oc - f)$ as $K \to
  \infty$). Similarly,
    \begin{align*}
      \pr \lbrace f' + \Phi(\boldsymbol\eta') &\in \Oc
      \rbrace = \lim_{K \to \infty} \int_{\Phi^{-1}_K (\Oc - f')}
      \Lc^K(\boldsymbol\eta_K'; \mathbf b_K) d
      \boldsymbol\eta_K'.
    \end{align*}
    By linearity of $\Phi_K$, we have $\Phi_K^{-1}(\Oc - f') =
    \Phi_K^{-1}(\Oc - f) + \boldsymbol \delta_K$ where $\boldsymbol
    \delta = \Phi^{-1}(f - f')$. Therefore,
  \begin{align*}
    \pr \lbrace f' + &\Phi(\boldsymbol\eta') \in \Oc \rbrace = \lim_{K
      \to \infty} \int_{\Phi^{-1}_K (\Oc - f)} \hspace{-15pt}
    \Lc^K(\boldsymbol\eta_K + \boldsymbol \delta_K; \mathbf b_K) d
    \boldsymbol\eta_K.
  \end{align*}
  Note that
  \begin{align*}
    \frac{\Lc^K(\boldsymbol\eta_K + \boldsymbol \delta_K; \mathbf
      b_K)}{\Lc^K(\boldsymbol\eta_K; \mathbf b_K)} = \prod_{k = 1}^K
    \frac{\Lc(\eta_k + \delta_k; b_k)}{\Lc(\eta_k; b_k)} \le
    e^{\sum_{k = 1}^K \frac{|\delta_k|}{b_k}}.
  \end{align*}
  After multiplying both sides by $\Lc^K(\boldsymbol\eta_K;
  \mathbf b_K)$, integrating over $\Phi^{-1}_K (\Oc - f)$, and
  letting $K \to \infty$, we have
  \begin{align*}
    \pr \lbrace f' + \Phi(\boldsymbol\eta') \in \Oc \rbrace \le
    e^{\sum_{k = 1}^\infty \frac{|\delta_k|}{b_k}} \pr \lbrace f +
    \Phi(\boldsymbol\eta) \in \Oc \rbrace.
  \end{align*}
  Finally, the coefficient of the exponential can be upper bounded
  using Holder's inequality with $p = q = 2$ as
  \begin{align*}
    &\sum_{k = 1}^\infty \frac{|\delta_k|}{b_k} = \sum_{k = 1}^\infty
    \frac{k^q |\delta_k|}{k^q b_k} \le \left(\sum_{k = 1}^\infty
      \frac{1}{(k^q b_k)^2} \right)^{\frac{1}{2}} \left(\sum_{k =
        1}^\infty (k^q \delta_k)^2 \right)^{\frac{1}{2}}
    \\
    &= \!\!\left(\sum_{k = 1}^\infty \frac{1}{(\gamma k^{q -
            p})^2} \right)^{\!\!\frac{1}{2}} \!\!\!\! \left\| f - f'
      \right\|_{\Vc_q}
    \!=\! \frac{1}{\gamma} \sqrt{\zeta(2 (q \!-\! p))} \left\| f - f'
    \right\|_{\Vc_q}\!,
  \end{align*}
  which completes the proof.
\end{proof}

\begin{remark}\longthmtitle{Choice of $q$} {\rm The choice of
    parameter~$q$ affects the trade-off between the size of the
    adjacency space~$\Vc_q$ and the noise required for
    privacy. From~\eqref{eq:Vq}, we see that decreasing $q$ makes
    $\Vc_q$ larger, allowing for the privacy preservation of a larger
    collection of functions. However, as expected, preserving privacy
    in a larger space requires more noise. From (12), $\gamma$ will be
    larger for a fixed $\epsilon$, (since $p$ cannot be decreased by
    the same amount as $q$ and $\zeta$ is monotonically decreasing),
    resulting in larger $b_k$ and larger noise. We show later in
    Theorem VI.2 that the guaranteed upper bound on the expected
    minimizer deviation also increases as $\{q_i\}_{i = 1}^n$
    decrease. \oprocend}
\end{remark}

\section{Differentially Private Distributed
  Optimization}\label{sec:diff-private-dist-opt}

In this section, we employ functional differential privacy to solve
the differentially private distributed optimization problem formulated
in Section~\ref{sec:prob-state} for a group of $n \in \intpos$ agents.
For convenience, we introduce the shorthand notation $ \Sc_0 = C^2(D)
\subset L_2(D)$ and, for given $\overline{u}>0$, $0 < \alpha < \beta$,
\begin{multline*}
  \Sc = \setdef{h \in \Sc_0}{|\nabla h(x)| \le \overline u, \; \forall
    x \in D
    \\
    \text{ and } \alpha I_d \le \nabla^2 h(x) \le \beta I_d, \;
    \forall x \in D^o} ,
\end{multline*}
for twice continuously differentiable functions with bounded gradients
and Hessians. In the rest of the paper, we assume that the agents' local
objective functions $f_1,\dots,f_n$ belong to~$\Sc$.

\subsection{Smoothness and Regularity of the Perturbed
  Functions}\label{sec:regularity}

We address here the challenge~(ii) laid out in
Section~\ref{sec:algorithm-design-plan}.  To exploit the framework of
functional differential privacy for optimization, we need to ensure
that the perturbed functions have the smoothness and regularity
properties required by the distributed coordination algorithm. In
general, the output of~\eqref{eq:f-hat} might neither be smooth nor
convex. We detail next how to address these problems by defining
appropriate maps that, when composed with $\Mc$ in~\eqref{eq:f-hat},
yield functions with the desired
properties. Proposition~\ref{prop:post-proc} ensures that differential
privacy is retained throughout this procedure.

\subsubsection{Ensuring Smoothness}\label{subsec:en-smooth}

To ensure smoothness, we rely on the fact that $\Sc_0$ is dense in
$L_2(D)$ and, therefore, given any function $g$ in $L_2(D)$, there
exists a smooth function arbitrarily close to it, i.e.,
\begin{align*}
  \forall \varepsilon > 0 , \; \exists \hat g^s \in \Sc_0 \quad
  \text{such that} \quad \|g - \hat g^s\| < \varepsilon.
\end{align*}
Here, $\varepsilon$ is a design parameter and can be chosen
sufficiently small (later, we show how to do this so that the accuracy
of the coordination algorithm is not affected).

  \begin{remark}\longthmtitle{Smoothening and Truncation}
    {\rm 
    A natural choice for the smoothening step, if the basis functions
    are smooth (i.e., $\{e_k\}_{k = 1}^\infty \subset \Sc_0$), is
    truncating the infinite expansion of~$g$. Such truncation is also
    inevitable in practical implementations due to the impossibility
    of handling infinite series.  The appropriate truncation order
    depends on the specific function, the basis set, and the noise
    decay rate ($p$ in~\eqref{eq:b}).
  } \oprocend
\end{remark}

\subsubsection{Ensuring Strong Convexity and Bounded
  Hessian}\label{subsec:en-convex}

The next result ensures that the orthogonal projection from $\Sc_0$
onto $\Sc$ is well defined, and can therefore be used to ensure strong
convexity and bounded Hessian of the perturbed functions.

\begin{proposition}\longthmtitle{Convexity of $\Sc$ and its closedness
    relative to $\Sc_0$}\label{prop:closure}
  The set $\Sc$ is convex and closed as a subset of $\Sc_0$ under the
  $2$-norm.
\end{proposition}
\begin{proof}
  The set $\Sc$ is clearly convex because, if $h_1, h_2 \in \Sc$ and
  $\lambda \in [0, 1]$, then for all $x \in D^o$,
  \begin{align*}
    \nabla^2((1 - \lambda) h_1(x) + \lambda h_2(x)) &= (1 - \lambda)
    \nabla^2 h_1(x) + \lambda \nabla^2 h_2(x)
    \\
    &\ge (1 - \lambda) \alpha I_d + \lambda \alpha I_d = \alpha I_d.
  \end{align*}
  Similarly, $\nabla^2((1 - \lambda) h_1(x) + \lambda h_2(x)) \le
  \beta I_d$. Also,
  \begin{align*}
    |\nabla ((1 - \lambda) h_1(x) + \lambda h_2(x))| &\le (1 -
    \lambda)|\nabla h_1(x)| + \lambda|\nabla h_2(x)|
    \\
    &\le (1 - \lambda) \overline u + \lambda \overline u \le \overline
    u,
  \end{align*}
  for all $x \in D$. To establish closedness, let
  \begin{align*}
    \Sc_1 &= \setdef{h \in \Sc_0}{\alpha I_d \le \nabla^2 h(x) \le
      \beta I_d, \; \forall x \in D^o},
    \\
    \Sc_2 &= \setdef{h \in \Sc_0}{|\nabla h(x)| \le \overline u, \;
      \forall x \in D }.
  \end{align*} 
  Since $\Sc = \Sc_1 \cap \Sc_2$, it is enough to show that $\Sc_1$
  and $\Sc_2$ are both closed subsets of $\Sc_0$.
  
  To show that $\Sc_1$ is closed, let $\{h_k\}_{k = 1}^\infty$ be a
  sequence of functions in $\Sc_1$ such that $h_k
  \xrightarrow{\|\cdot\|_2} h \in \Sc_0$. We show that $h \in
  \Sc_1$. Since $h_k - \frac{\alpha}{2} |x|^2 \xrightarrow{\|\cdot\|_2}
  h - \frac{\alpha}{2} |x|^2$ and $L_2$ convergence implies pointwise
  convergence of a subsequence almost everywhere, there exists
  $\{h_{k_\ell}\}_{\ell = 1}^\infty$ and $Y \subset D$ such that $m(D
  \setminus Y) = 0$ and $h_{k_\ell}(x) -
  \frac{\alpha}{2} |x|^2 \to h(x) - \frac{\alpha}{2} |x|^2$ for all $x
  \in Y$. It is straightforward to verify that $Y$ is dense in $D$ and
  therefore $Y \cap D^o$ is dense in $D^o$.  Then, by~\cite[Theorem
  10.8]{RTR:70}, $h - \frac{\alpha}{2} |x|^2$ is convex on $D^o$, so
  $\alpha I_2 \le \nabla^2 h(x)$ for all $x \in D^o$. Similarly, one
  can show that $\nabla^2 h(x) \le \beta I_d$ for all $x \in
  D^o$. Therefore, $h \in \Sc_1$.
  
  The proof of closedness of $\Sc_2$ is more involved. By
  contradiction, assume that $\{h_k\}_{k = 1}^\infty$ is a sequence of
  functions in $\Sc_2$ such that $h_k \xrightarrow{\|\cdot\|_2} h \in
  \Sc_0$ but $h \notin \Sc_2$. Therefore, there exist $x_0 \in D^o$
  such that $|\nabla h(x_0)| > \overline u$ and, by continuity of
  $\nabla h$, $\delta_0 > 0$ and $\upsilon_0 > 0$ such that
  \begin{align*}
    |\nabla h(x)| \ge \overline u + \upsilon_0, \quad \forall x \in
    B(x_0, \delta_0) \subseteq D .
  \end{align*}
  Let $ u_0 = \frac{\nabla h(x_0)}{|\nabla h(x_0)|}$.  By continuity
  of $\nabla h$, for all $\upsilon_1 > 0$ there exists $\delta_1 \in
  (0, \delta_0]$ such that
  \begin{align*}
    \nabla h(x) \cdot u_0 \ge (1 - \upsilon_1) |\nabla h(x)| , \quad
    \forall x \in B(x_0,\delta_1) .
  \end{align*}
  As mentioned above, $L_2$ convergence implies pointwise convergence
  of a subsequence $\{h_{k_\ell}\}_{\ell = 1}^\infty$ almost
  everywhere. In turn, this subsequence converges to $h$ almost
  uniformly, i.e., for all $\upsilon_2 > 0$ and all $\upsilon_3 > 0$,
  there exist $E \subset D$ and $L \in \intpos$ such that $m(E) <
  \upsilon_2$ and
  \begin{align}\label{eq:almost-unif-conv}
    |h_{k_\ell}(x) - h(x)| < \upsilon_3 , \quad \forall x \in D
    \setminus E \text{ and } \ell \ge L.
  \end{align}
  For ease of notation, let $\delta_2 = \delta_1/2$. Using the
  fundamental theorem of line integrals~\cite{REW-HFT:03},
  for all $x \in B(x_0,\delta_2) \setminus E$,
  \begin{align}\label{eq:lbound}
    & h(x + \delta_2 u_0) - h(x) = \int_x^{x + \delta_2 u_0} \nabla h
    \cdot d r \notag
    \\
    &\quad = \int_x^{x + \delta_2 u_0} \nabla h \cdot u_0 |d r| \ge
    \int_x^{x + \delta_2 u_0} (1 - \upsilon_1) |\nabla h| |d r| \notag
    \\
    &\quad \ge (1 - \upsilon_1)(\overline u + \upsilon_0) \delta_2.
  \end{align}
  Similarly, for all $x \in B(x_0,\delta_2) \setminus E$ and all $\ell
  \in \intpos$,
  \begin{align}\label{eq:ubound}
    \notag h_{k_\ell}(x + \delta_2 u_0) - h_{k_\ell}(x) &= \int_x^{x +
      \delta_2 u_0} \nabla h_{k_\ell} \cdot d r
    \\
    &\le \int_x^{x + \delta_2 u_0} |\nabla h_{k_\ell}| |d r| \le
    \overline u \delta_2.
  \end{align}
  Putting~\eqref{eq:lbound},~\eqref{eq:ubound},
  and~\eqref{eq:almost-unif-conv} together and choosing $\upsilon_3 =
  \upsilon_1 \delta_2 \overline u$, we have for all $x \in
  B(x_0,\delta_2) \setminus E$ and all $\ell \ge L$,
  \begin{align}\label{eq:lubound}
    h(x + \delta_2 u_0) - &h_{k_\ell}(x + \delta_2 u_0)
    \\
    \notag &\ge h(x) - h_{k_\ell}(x) + \delta_2(1 -
    \upsilon_1)(\overline u + \upsilon_0) - \delta_2 \overline u
    \\
    &\ge \delta_2(1 - \upsilon_1)(\overline u + \upsilon_0) -
    \delta_2(1 + \upsilon_1) \overline u \triangleq \upsilon_4 .
    \notag
  \end{align}
  The quantity $\upsilon_4$ can be made strictly positive choosing
  $\upsilon_1 = \frac{\upsilon_0}{4 \overline u + 3 \upsilon_0} >
  0$. Let $E^+ = E + \delta_2 u_0$ and $x_1 = x_0 + \delta_2
  u_0$. Then, \eqref{eq:lubound} can be rewritten as
  \begin{align*}
     h(x) - h_{k_\ell}(x) \ge \upsilon_4 , \quad \forall x \in
    \Nc_{\delta_2}(x_1) \setminus E^+ \text{ and } \ell \ge L ,
  \end{align*}
  which, by choosing $\upsilon_2 = \frac{1}{2} m(B(x_1,\delta_2))$, implies
  \begin{align*}
    & \int_{\Nc_{\delta_2}(x_1) \setminus E^+} |h(x) -
    h_{k_\ell}(x)|^2 d x \ge \upsilon_4^2 \cdot m(B(x_1,\delta_2)
    \setminus E^+)
    \\
    &\Rightarrow \|h - h_{k_\ell}\| \ge \upsilon_4
    \sqrt{m(B(x_1,\delta_2))/2}
    > 0.
  \end{align*}
  This contradicts $h_{k_\ell} \xrightarrow{\|\cdot\|_2} h$, so $\Sc_2$
  must be closed.
\end{proof}

Given the result in Proposition~\ref{prop:closure}, the best
approximation in $\Sc$ of a function $h \in \Sc_0$ is its unique
projection onto $\Sc$, i.e.,
\begin{align*}
  \tilde h = \proj_\Sc (h).
\end{align*}
By definition, $\tilde h$ has bounded gradient and Hessian.

\subsection{Algorithm Design and Analysis}\label{subsec:alg-des-anal}

We address here the challenge~(iii) laid out in
Section~\ref{sec:algorithm-design-plan} and put together the
discussion above to propose a class of differentially private,
distributed optimization algorithms that solve Problem~\ref{problem}.
Unlike the message-perturbing algorithms where agents use the original
objective functions in the computations and rely on perturbing the
inter-agent messages, here we propose that agents locally perturb
their objective functions and use them in their computations, without
adding any additional noise to the inter-agent messages. Therefore, we
require each agent $i \in \{1, \dots, n \}$ to first compute
\begin{subequations}\label{eq:perturb}
  \begin{align}\label{eq:perturb-I}
    \hat f_i = \Mc(f_i, \boldsymbol{\eta}_i) = f_i +
    \Phi(\boldsymbol{\eta}_i) ,
  \end{align}
  where $\boldsymbol{\eta}_i$ is a sequence of Laplace noise generated
  by~$i$ according to~\eqref{eq:eta} with the choice~\eqref{eq:b}, then
  select $\hat f^s_i \in \Sc_0$ such that
  \begin{align}\label{eq:perturb-II}
    \|\hat f_i - \hat f^s_i\| < \varepsilon_i,
  \end{align}
  and finally compute
  \begin{align}\label{eq:perturb-III}
    \tilde f_i = \proj_\Sc (\hat f^s_i) .
  \end{align}
\end{subequations}
After this process, agents participate in \emph{any} distributed
optimization algorithm with the modified objective functions $\{\tilde
f_i\}_{i = 1}^n$.  Let
\begin{align*}
  \tilde{x}^* = \argmin_{x \in X} \sum_{i = 1}^n \tilde f_i \quad
  \text{and} \quad x^* = \argmin_{x \in X} \sum_{i = 1}^n f_i ,
\end{align*}
denote, respectively, the output of the distributed algorithm and the
optimizer for the original optimization problem (with objective
functions $\{f_i\}_{i = 1}^n$).  The following result establishes the
connection between the algorithm's accuracy and the design parameters.

\begin{theorem}\longthmtitle{Accuracy of the proposed class of distributed,
    differentially private coordination algorithms}\label{thm:accuracy}
  Consider a group of $n$ agents which perturb their local objective
  functions according to~\eqref{eq:perturb} with Laplace
  noise~\eqref{eq:eta} of variance~\eqref{eq:b}, where $q_i > 1$,
  $\gamma_i > 0$, and $p_i \in \left (\frac{1}{2}, q_i - \frac{1}{2}
  \right)$ for all $i \in \until{n}$.  Let the agents participate in
  any distributed coordination algorithm that asymptotically converges
  to the optimizer~$\tilde{x}^*$ of the perturbed aggregate objective
  function. Then, $\epsilon_i$-differential privacy of each agent
  $i$'s original objective function is preserved with $\epsilon_i =
  \sqrt{\zeta(2(q_i - p_i))}/ \gamma_i$ and
  \begin{align*}
    \E \left|\tilde{x}^* - x^* \right| \le \sum_{i = 1}^n
    \kappa_{n} \left(\gamma_i \sqrt{\zeta(2 p_i)} \right)
    + \kappa_{n}(\varepsilon_i),
  \end{align*}
  where the function $\kappa_n \equiv \kappa_{n\alpha,n\beta} \in \Kc_\infty$ is
  defined in Proposition~\ref{prop:argmin-lip}.
\end{theorem}
\begin{proof}
  Since the distributed algorithm is a post-processing step on the
  perturbed functions, privacy preservation of the objective functions
  follows from Theorem~\ref{thm:diff-privacy} and
  Proposition~\ref{prop:post-proc}.  For convenience, let
  \begin{align*}
    \Delta = \E \left|\tilde{x}^* - x^* \right| = \E \Big|
    \argmin_{x \in X} \sum_{i = 1}^n \tilde f_i - \argmin_{x \in X}
    \sum_{i = 1}^n f_i \Big|.
  \end{align*}
  Since $\mu_{n\alpha,n\beta}$ is convex and belongs to
  class~$\Kc_\infty$ (so is monotonically increasing), $\kappa_n$ is
  concave and belongs to class $\Kc_\infty$ and so is
  subadditive. Therefore, using Proposition~\ref{prop:argmin-lip},
  \begin{align*}
    \Delta &\le \E \Big[ \kappa_n \Big( \Big\| \sum_{i = 1}^n \tilde
    f_i - \sum_{i = 1}^n f_i \Big\| \Big) \Big]
    \\
    &\le \E \Big[\kappa_n \Big( \sum_{i = 1}^n \| \tilde f_i - f_i \|
    \Big) \Big] \le \sum_{i = 1}^n \E \big[\kappa_n \big( \|\tilde f_i
    - f_i \|\big) \big].
  \end{align*}
 Then, by the non-expansiveness of projection, we have
  \begin{align}\label{eq:delta-bound-mid}
    \notag \Delta &\le \sum_{i = 1}^n \E \big[ \kappa_n \big(\| \hat
    f_i^s - f_i \|\big) \big]
    \\
    \notag &\le \sum_{i = 1}^n \E \big[ \kappa_n \big(\| \hat f_i^s -
    \hat f_i \|\big) + \kappa_n \big(\| \hat f_i - f_i \|\big) \big]
    \\
    &\le \sum_{i = 1}^n \big( \kappa_n (\varepsilon_i) + \E \big[
    \kappa_n \big(\| \boldsymbol{\eta}_i \| \big) \big]\big).
  \end{align}
  By invoking Jensen's inequality twice, for all $i \in \until{n}$,
  \begin{align}\label{eq:jensen}
    &\E \big[ \kappa_n (\| \boldsymbol{\eta}_i \| ) \big] \le \kappa_n
    \big(\E \big[ \| \boldsymbol{\eta}_i \| \big] \big) =
    \kappa_n \big(\E \big[\sqrt{\| \boldsymbol{\eta}_i
      \|^2} \big] \big) \\
    \notag &\!\!\le \kappa_n\Big(\sqrt{\E \big[\|\boldsymbol{\eta}_i\|^2
      \big]} \Big) \!=\! \kappa_n\Big(\sqrt{\sum\nolimits_{k = 1}^\infty
      b_{i, k}^2} \Big) \!=\! \kappa_n \Big(\gamma_i \sqrt{\zeta(2 p_i)} \Big).
  \end{align}
  The result follows from~\eqref{eq:delta-bound-mid}
  and~\eqref{eq:jensen}.
\end{proof}

The following result describes the trade-off between accuracy and
privacy. The proof follows by direct substitution.

\begin{corollary}\longthmtitle{Privacy-accuracy
    trade-off}\label{cor:pri-acc-rel}
  Under the hypotheses of Theorem~\ref{thm:accuracy}, if $p_i =
  \frac{q_i}{2}$ in~\eqref{eq:b} for all $i$, then
  \begin{align}\label{eq:pri-acc-rel}
     \E \left|\tilde{x}^* - x^* \right| \le \sum_{i = 1}^n \kappa_n
    \left(\frac{\zeta(q_i)}{\epsilon_i} \right) +
    \kappa_n(\varepsilon_i).
  \end{align}
\end{corollary}

In Corollary~\ref{cor:pri-acc-rel}, $q_i$ and $\epsilon_i$ are chosen
independently, which in turn determines the value of $\gamma_i$
according to~\eqref{eq:eps}.  Also, it is clear
from~\eqref{eq:pri-acc-rel} that in order for the accuracy of the
coordination algorithm not to be affected by the smoothening step,
each agent $i \in \until{n}$ has to take the value of $\varepsilon_i$
sufficiently small so that it is negligible relative to $\zeta(2p_i) /
\epsilon_i$. In particular, this procedure can be executed for any
arbitrarily large value of $\epsilon_i$, so that in the case of no
privacy requirements at all, perfect accuracy is recovered, as
specified in Problem~\ref{problem}.

\begin{remark}\longthmtitle{Accuracy bound for
    sufficiently large domains} 
  {\rm One can obtain a less conservative bound
    than~\eqref{eq:pri-acc-rel} on the accuracy of the proposed class
    of algorithms if the minimizers of all the agents' objective
    functions are sufficiently far from the boundary of $X$. This can
    be made precise via Corollary~\ref{cor:large-domain}. If the
    aggregate objective function satisfies~\eqref{eq:uline-rD} and the
    amount of noise is also sufficiently small so that the minimizer
    of the sum of the perturbed objective functions satisfies this
    condition, then invoking Corollary~\ref{cor:large-domain}, we have
    \begin{align*}
       \E \left|\tilde{x}^* - x^* \right| &\le \frac{L}{n^2} \sum_{i =
        1}^n \left(\gamma_i^\frac{2}{d + 4} \zeta(2 p_i)^\frac{1}{d +
          4} + \varepsilon_i^\frac{2}{d + 4} \right)
      \\
      &= \frac{L}{n^2} \sum_{i = 1}^n
      \left[\left(\frac{\zeta(q_i)}{\epsilon_i}\right)^\frac{2}{d + 4}
        + \varepsilon_i^\frac{2}{d + 4} \right],
  \end{align*}
  where the equality holds under the assumption that $p_i =
  \frac{q_i}{2}$ in~\eqref{eq:b} for all $i \in \until{n}$. 
}
 \oprocend
\end{remark}

\section{Simulations}\label{sec:sims}

In this section, we report simulation results of our algorithm for
Example~\ref{example} with $D = X = [-5, 5]^2$, $n = 10$, $N_d = 100$,
and $\lambda = 0.01$.  The orthonormal basis of $L_2(D)$ is
constructed from the Gram-Schmidt orthogonalization of the Taylor
functions and the series is truncated to the second, sixth, and
fourteenth orders, resulting in $15$, $28$, and $120$-dimensional
coefficient spaces, respectively. This truncation also acts as the
smoothening step described in Section~\ref{subsec:en-smooth}, where
higher truncation orders result in smaller~$\varepsilon$.  We evaluate
the projection operator in~\eqref{eq:perturb-III} by numerically
solving the convex optimization problem $\min_{\tilde f_i \in \Sc}
\|\tilde f_i - \hat f_i^s\|$, where $\hat f_i^s$ is the result of the
truncation. The parameters of $\Sc$ are given by $\alpha = N_d
\lambda$, $\beta = N_d \lambda + N_d r_D \sqrt2 + e^{2r_D}$, and
$\overline u = \sqrt2 N_d (\lambda r_D + e^{2r_D})$ where $r_D =
5$. Rather than implementing any specific distributed coordination
algorithm, we use an iterative interior-point algorithm on $\tilde f$
and $f$ to find the perturbed $\tilde{x}^*$ and original $x^*$
optimizers, respectively (these points correspond to the asymptotic
behavior of any provably correct distributed optimization algorithm
with the perturbed and original functions, respectively).

\begin{figure}[htb]
  \centering
  \includegraphics[width = \linewidth]{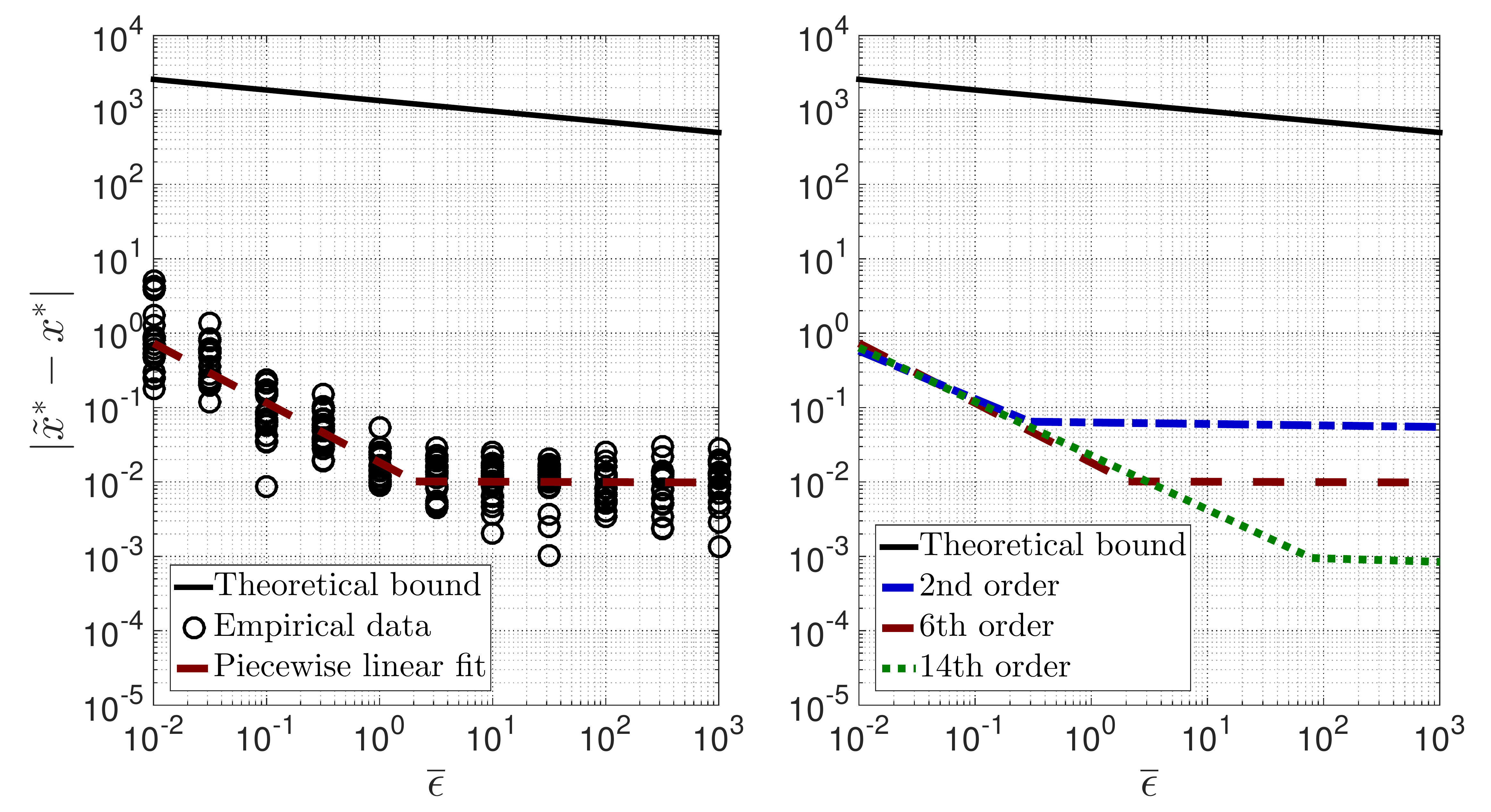}
  \caption{Privacy-accuracy trade-off curve of the
      proposed class of distributed, differentially private algorithms
      in Section~\ref{subsec:alg-des-anal} for Example~\ref{example}
      (with the same data as Figure~\ref{fig:huang}) and different
      truncation orders. Left: empirical data and its best piecewise
      linear fit for 6th-order truncation of the function expansions,
      together with the theoretical upper bound of
      Corollary~\ref{cor:pri-acc-rel}. Right: piecewise linear fit of
      empirical data for 2nd, 6th, and 14th order truncations as well
      as the theoretical upper bound. Accuracy improves with the
      truncation order.}\label{fig:sweep}
\end{figure}

The privacy levels are chosen the same for all agents, i.e., $\epsilon
= (\overline \epsilon, \overline \epsilon, \cdots, \overline
\epsilon)$, and $\overline \epsilon$ is swept logarithmically over
$[10^{-2}, 10^3]$. For each $i \in \until{n}$, we set $q_i = 2p_i =
1.1$ and $\gamma_i = \sqrt{\zeta(2(q_i - p_i))} / \overline
\epsilon$. For each value of $\overline \epsilon$ and truncation
order, the simulations are repeated 20 times to capture the
stochasticity of the solutions.  Figure~\ref{fig:sweep} illustrates
the error $|\tilde{x}^* - x^*|$ as a function of $\overline \epsilon$
for different truncation orders, together with the best linear fit of
$\log |\tilde{x}^* - x^*|$ against $\log \overline \epsilon$, and the
upper bound obtained in Corollary~\ref{cor:pri-acc-rel}. The
conservative nature of this upper bound can be explained by noting the
approximations leading to the computation of $\kappa_n$ in
Proposition~\ref{prop:argmin-lip}, suggesting that there is room for
refining this bound. Figure~\ref{fig:sweep} shows that accuracy keeps
improving as the privacy requirement is relaxed until the
$\varepsilon$-term (resulting from the smoothening/truncation)
dominates the error. This ``saturation'' value can be decreased by
increasing the truncation order (which comes at the expense of more
computational complexity), in contrast with the behavior of
message-perturbing algorithms, cf. Figure~\ref{fig:huang}.  It is
important to mention that the respective error values for a fixed
$\epsilon$ cannot be compared between Figures~\ref{fig:huang}
and~\ref{fig:sweep} because, in~\cite{ZH-SM-NV:15}, $\epsilon$ is
defined as the total exponent in~\eqref{eq:diff-privy}, i.e.,
$\epsilon_{i_0} \|f_{i_0} - f_{i_0}'\|_\Vc$. However, it can be seen
that the accuracy in Figure~\ref{fig:huang} is almost indifferent to
the value of $\epsilon$ and is in the same order as $r_D = 5$. This is
explained by the impossibility result of Proposition IV.2: since the
noise-free algorithm of~\cite{ZH-SM-NV:15} is not asymptotically
stable, depending on the specific application, its accuracy may not be
desirable regardless of the value of~$\epsilon$. In contrast, the
accuracy in Figure~\ref{fig:sweep} keeps improving as $\epsilon$ is
increased (with an appropriate choice of truncation order).

\section{Conclusions and Future Work}\label{sec:conclusion}

We have studied the distributed optimization of the sum of local
strongly convex objective functions by a group of agents subject to
the requirement of differential privacy.  We have established the
incompatibility between differential privacy and the asymptotic
stability of the underlying (noise-free) dynamics for coordination
strategies based on the perturbation of the messages among agents.
This has motivated us to develop a new framework for differentially
private release of functions from the $L_2$ space that can be applied
to arbitrary separable Hilbert spaces. We have also carefully
described how perturbed functions with the desired regularity
requirements can be provided to any distributed coordination algorithm
while preserving privacy. Finally, we upper bounded the accuracy of
the resulting class of distributed, differentially private
coordination algorithms.  Future work will analyze how to relax the
smoothness and convexity assumptions on the local objective functions
and the compactness hypothesis on their domains, characterize the
\emph{optimal} privacy-accuracy trade-off curve of distributed
coordination algorithms based on objective perturbation, characterize
the expected suboptimality gap of distributed coordination algorithms
for a given desired level of privacy, provide procedures for choosing
the truncation order of the functional expansion based on the
objective function properties (which must itself be differentially
private), compare the numerical efficiency of different orthonormal
bases for the $L_2$ space, and further understand the appropriate
scale of the privacy parameter for specific application domains.

\section*{Acknowledgments}
The authors would like to thank the anonymous reviewers for helpful
comments and suggestions that helped improve the presentation. This
research was partially supported by NSF Award CNS-1329619 and Award
FA9550-15-1-0108.

\appendices

\section{$\Kc$-Lipschitz Property of the $\argmin$
  Map}\label{sec:argmin-lip}

Here, we establish the Lipschitzness of the $\argmin$ map under
suitable assumptions.  This is a strong result of independent interest
given that $\argmin$ is not even continuous for arbitrary $C^2$
functions. Our accuracy analysis for the proposed class of
distributed, differentially private algorithms in
Section~\ref{subsec:alg-des-anal} relies on this result. We begin with
a lemma stating a geometric property of balls contained in
convex, compact domains.

\begin{lemma}\longthmtitle{Minimum portion of balls contained in
    convex compact domains}\label{lem:portion}
  Assume $D \subset \real^d$ is convex, compact, and has nonempty
  interior and let $r_D > 0$ denote its inradius. Then, there exists
  $\lambda_D \in (0, 1)$ such that,
  \begin{align*}
    m(B(x, r) \cap D) \ge \lambda_D m(B(x, r)) ,
  \end{align*}
  for any $x \in D$ and $r \le r_D$.
\end{lemma}
\begin{proof}
  \renewcommand{\qedsymbol}{} Let $B(c_D, r_D)$ be the inball of $D$,
  i.e., the largest ball contained in $D$. If this ball is not unique,
  we pick one arbitrarily. Since $D^o \neq \emptyset$, $r_D > 0$. Let
  $R_D$ be the radius of the largest ball centered at $c_D$ that
  contains $D$. Since $D$ is compact, $R_D < \infty$. For any $x \in
  D$ that is on or outside of $B(c_D, r_D)$, let $\Sigma$ be the
  intersection of $B(c_D, r_D)$ and the hyperplane passing through
  $c_D$ and perpendicular to $c_D - x$. Consider the cone $C =
  \text{conv}(\Sigma \cup \{x\})$ where $\text{conv}$ denotes convex
  hull. Since $D$ is convex, $C \subseteq D$. Note that $C$ has half
  angle $\theta_x = \tan^{-1} \frac{r_D}{|x - c_D|}$ so the solid
  angle at its apex is
  \begin{align}\label{eq:Omega_x}
    \Omega_{\theta_x} = \frac{2 \pi^\frac{d - 1}{2}}{\Gamma(\frac{d -
        1}{2})} \int_0^{\theta_x} \sin^{d - 2}(\phi) d \phi.
  \end{align}
  Therefore, for any $r \le r_D$, the proportion
  $\frac{\Omega_{\theta_x}}{\Omega_d}$ of $B(x, r)$ is contained in
  $D$ where $\Omega_d = \frac{2 \pi^\frac{d}{2}}{\Gamma(\frac{d}{2})}$
  is the total $d$-dimensional solid angle.  For any $x$ inside
  $B(c_D, r_D)$, the same argument holds with
  \begin{align*}
    \theta_x = \max_{|x - c_D| \ge r_D} \tan^{-1} \frac{r_D}{|x -
      c_D|} = \frac{\pi}{4}.
  \end{align*}
  Therefore, for arbitrary $x \in D$, the statement holds with
  \begin{align*}
    \lambda_D = \min_{x \in D} \frac{\Omega_{\theta_x}}{\Omega_d} =
    \frac{1}{\Omega_d} \Omega_{\tan^{-1}(r_D / R_D)}. \eqoprocend
  \end{align*}
\end{proof}

We are now ready to establish the $\Kc$-Lipschitzness of the $\argmin$ map.

\begin{proposition}\longthmtitle{$\Kc$-Lipschitzness of
    $\argmin$}\label{prop:argmin-lip}
  For any two functions $f, g \in \Sc$,
  \begin{align}\label{eq:argmin-lip}
    \big| \argmin_{x \in X} f - \argmin_{x \in X} g \big| \le
    \kappa_{\alpha,\beta}(\|f - g\|),
  \end{align}
  where $\kappa_{\alpha,\beta} \in \Kc_\infty$ is given by
  \begin{align*}
    \kappa_{\alpha,\beta}^{-1}(r) = \frac{\alpha^2 \pi^\frac{d}{2}}{d
      2^{d + 3} \Gamma(\frac{d}{2})} \lambda_D
    \left(\frac{r_D}{d_D}\right)^d r^4 \mu_{\alpha,\beta}^d(r), \quad
    \forall r \in [0, \infty),
  \end{align*}
  $r_D$ and $\lambda_D$ are as in Lemma~\ref{lem:portion}, $d_D$ is
  the diameter of $D$, and $\mu_{\alpha,\beta} \in \Kc_\infty$ is
  defined for all $r \in [0, \infty)$ by
  \begin{align*}
    \mu_{\alpha,\beta}(r) = \frac{\alpha r^2}{2 \sqrt{\alpha \beta r^2
        + 2(\beta + \alpha) \overline u r + 4 \overline u^2}}.
  \end{align*}
\end{proposition}
\begin{proof}
  We consider the case where $a = \argmin_{x \in X} f(x) \neq
  \argmin_{x \in X} g(x) = b$ since the statement is trivial
  otherwise. Let $m_a = f(a)$, $m_b = g(b)$, $m = m_a - m_b$, $u_a =
  \nabla f(a)$, and $u_b = \nabla g(b)$. Without loss of generality,
  assume $m \ge 0$. Define,
  \begin{align*}
    f_l(x) &= \frac{\alpha}{2}|x - a|^2 + u_a^T(x - a) + m_a,
    \\
    g_u(x) &= \frac{\beta}{2}|x - b|^2 + u_b^T(x - b) + m_b,
  \end{align*}
  for all $x \in D$.  Since $f, g \in \Sc$, we can integrate $\nabla^2
  f \ge \alpha I_d$ and $\nabla^2 g \le \beta I_d$
  twice to get,
  \begin{align}\label{eq:fl-gu-ineq}
    \forall x \in D \quad f_l(x) \le f(x) \ \text{and} \ g(x) \le
    g_u(x) .
  \end{align}
  It follows that, for all $x \in D$,
  \begin{align*}
     |f(x) - g(x)| \ge [f_l(x) - g_u(x)]^+ \ge [f_l(x) - g_u(x) - m]^+,
   \end{align*}
  where $[z]^+ = \max\{z, 0\}$ for any $z \in \real$. After some
  computations, we get
  \begin{align*}
    f_l(x) - g_u(x) - m = - \frac{\beta - \alpha}{2} \Big( \left | x -
      c \right |^2 - r^2 \Big) ,
  \end{align*}
  where
  \begin{align*}
    c &= \frac{\beta b - \alpha a + u_a - u_b}{\beta - \alpha},
    \\
    r^2 &= \frac{\alpha \beta |a - b|^2}{(\beta - \alpha)^2}+ \frac{
      |u_a - u_b|^2}{(\beta - \alpha)^2} + \frac{2 (\beta u_a - \alpha
      u_b)^T(b - a)}{(\beta - \alpha)^2}.
  \end{align*}
  Therefore, the region where $f_l - g_u - m \ge 0$ is $B(c, r)$.
  Next, we seek to identify a subset inside this ball where we can
  determine a strictly positive lower bound of $f_l - g_u$ that
  depends on the difference $|a-b|$. To this effect, note that $b \in
  B(c, r)$, since
  \begin{align*}
    r^2 - |c - b|^2 = \frac{\alpha}{\beta - \alpha} |a - b|^2 +
    \frac{2}{\beta - \alpha} u_a^T (b - a),
  \end{align*}
  and, by the convexity of the problem, $u_a^T (b - a) \ge 0$.  Let
  $\underline r = r - |c - b| > 0$ be the radius of the largest ball
  centered at $b$ and contained in $B(c, r)$. We have,
  \begin{align*}
    &r^2 - |c - b|^2 = (r - |c - b|)(r + |c - b|) \ge
    \frac{\alpha}{\beta - \alpha} |a - b|^2
    \\
    &\Rightarrow \underline r \ge \frac{\frac{\alpha}{\beta - \alpha}
      |a - b|^2}{r + |c - b|} \ge \frac{\frac{\alpha}{\beta - \alpha}
      |a - b|^2}{2r} \ge \mu_{\alpha,\beta}(|a - b|),
  \end{align*}
  where in the last inequality, we have used $|u_a|, |u_b| \le
  \overline{u}$.  Next, note that for all $x \in B(c, \frac{r + |c -
    b|}{2})$,
  \begin{multline}\label{eq:fl-gu-m-lb}
    f_l(x) - g_u(x) - m \ge
    \\
    - \frac{\beta - \alpha}{2} \Big( \frac{r^2 + |c - b|^2 + 2 r |c -
      b|}{4} - r^2 \Big) .
  \end{multline}
  Using the bound $2 r |c - b| \le r^2 + |c - b|^2$, we get after some
  simplifications,
  \begin{align*}
    (f_l - g_u)(x) - m \ge \frac{\alpha}{4} |a - b|^2 +
    \frac{1}{2} u_a^T(b - a) \ge \frac{\alpha}{4} |a - b|^2 ,
  \end{align*}
  for all $x \in B(b, \frac{\underline r}{2}) \subset B(c, \frac{r +
    |c - b|}{2})$.  Therefore,
  \begin{align*}
    \|f - g\|^2 &= \int_D |f(x) - g(x)|^2 d x
    \\
    &\ge \int_D ([f_l(x) - g_u(x) - m]^+)^2 d x
    \\
    & \ge \int_{B(b, \frac{\underline r}{2}) \cap D} (f_l(x) - g_u(x)
    - m)^2 d x
    \\
    & \ge \frac{\alpha^2}{16} |a - b|^4 m\left(B(b, \tfrac{\underline
        r}{2}) \cap D\right)
    \\
    & \ge \frac{\alpha^2}{16} |a - b|^4 m\left(B(b,
      \tfrac{\mu_{\alpha,\beta}(|a - b|)}{2}) \cap D\right).
  \end{align*}
  Now, we invoke Lemma~\ref{lem:portion} to lower bound the last term.
  Note that $\mu_{\alpha,\beta}(|a - b|) \le |a - b| \le d_D$ for all
  $a, b \in D$. Therefore, $\frac{r_D}{d_D}
  \frac{\mu_{\alpha,\beta}(|a - b|)}{2} \le \min\{r_D,
  \mu_{\alpha,\beta}(|a - b|) / 2\}$, so by Lemma~\ref{lem:portion},
  \begin{align*}
    \|f - g\|^2 &\ge \frac{\alpha^2}{16} |a - b|^4 m\left(B\left(b,
        \frac{r_D \mu_{\alpha,\beta}(|a - b|}{2 d_D}\right) \cap
      D\right)
    \\
    &\ge \frac{\alpha^2}{16} |a - b|^4 \lambda_D \frac{2
      \pi^\frac{d}{2}}{d \Gamma(\frac{d}{2})} \frac{r_D^d}{2^d d_D^d}
    (\mu_{\alpha,\beta}(|a - b|))^d,
  \end{align*}
  which yields~\eqref{eq:argmin-lip}.
\end{proof}

The next result shows that if the minimizers of $f$ and $g$ are
sufficiently far from the boundary of $D$, then their gradients need
not be uniformly bounded and yet one can obtain a less conservative
characterization of the $\Kc$-Lipschitz property of the $\argmin$ map.

\begin{corollary}\longthmtitle{$\Kc$-Lipschitzness of
    $\argmin$ for sufficiently large domains}\label{cor:large-domain}
  If $f$ and $g$ belong to $ \Sc_1 = \setdef{h \in \Sc_0}{\alpha I_d
    \le \nabla^2 h(x) \le \beta I_d, \; \forall x \in D^o}$ and
  \begin{align}\label{eq:uline-rD}
    \argmin_{x \in X} f(x), \; \argmin_{x \in X} g(x) \in B(c_D,
    \underline r_D) \cap X^o,
  \end{align}
  where $\underline r_D = \frac{\beta - \alpha}{\alpha + \beta + 2
    \sqrt{\alpha \beta}} r_D$ and $B(c_D, r_D) \subset D$, then
  \begin{align*}
    \big| \argmin_{x \in X} f - \argmin_{x \in X} g \big| \le L \|f
    - g\|^{\frac{2}{d + 4}},
  \end{align*}
  where
  \begin{align*}
    L = \frac{d(d + 2)(d + 4) (\beta - \alpha)^{d + 2} \Gamma(d/2)}{4
      (\alpha \beta)^{d/2 + 2} \pi^{d/2}}.
  \end{align*}
\end{corollary}
\begin{proof}
  The proof follows in the same lines as the proof of
  Proposition~\ref{prop:argmin-lip} (and we use here the same
  notation).  Since the minimizers of $f$ and $g$ lie in the interior
  of $X$, $u_a = u_b = 0$. The main difference here is that due
  to~\eqref{eq:uline-rD}, we have for all $x \in B(c, r)$ that
  \begin{align*}
    |x - c_D| &\le |x - c| + |c - b| + |b - c_D|
    \\
    &\le \frac{\alpha + \sqrt{\alpha \beta}}{\beta - \alpha} 2
    \underline r_D + \underline r_D = r_D,
  \end{align*}
  so $B(c, r) \subset D$. Therefore, one can integrate $(f_l - g_u -
  m)^2$ on the whole $B(c, r)$ instead of its lower
  bound~\eqref{eq:fl-gu-m-lb} on the smaller ball $B(c, \frac{r + |c -
    b|}{2})$.  To explicitly calculate the value of the resulting
  integral, one can use the change of variables $x_i = c_i + r
  y_i^{1/2}, i \in \until{d}$ and then use the formula
  \begin{align*}
    \forall a_i > -1, \; \int_{S_d} y_1^{a_1} \cdots y_d^{a_d} dy =
    \frac{\Gamma(a_1 + 1) \cdots \Gamma(a_d + 1)}{\Gamma(a_1 + \dots +
      a_d + d + 1)},
  \end{align*}
  where $S_d = \setdef{y \in \real^d}{\sum_{i=1}^d y_i = 1 \text{ and
    } y_i \ge 0, \; i \in \until{d}} $. The result then follows from
  straightforward simplifications of the integral.
\end{proof}

\begin{IEEEbiography}[{\includegraphics[width=1in, height=1.25in, clip, keepaspectratio]{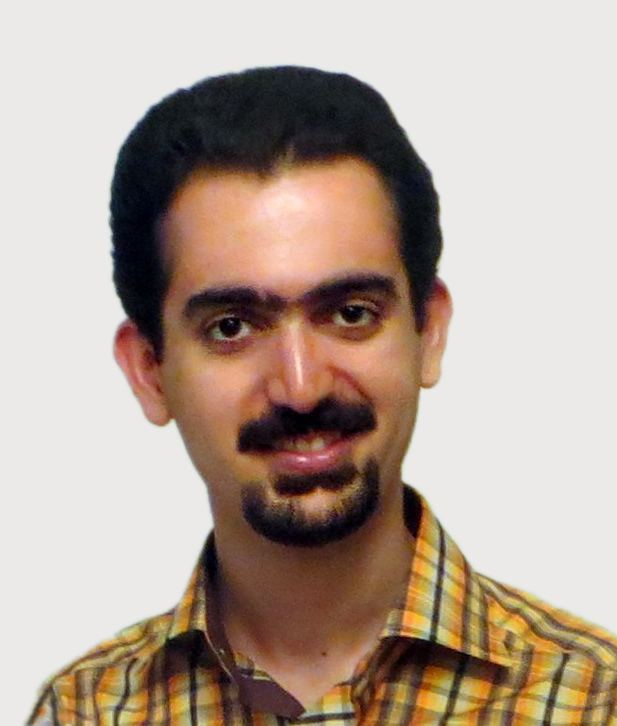}}]{Erfan Nozari}
  received his B.Sc. degree in Electrical Engineering-Controls from
  Isfahan University of Technology, Isfahan, Iran in 2013 and M.Sc. in
  Mechanical Engineering from University of California, San Diego, CA,
  USA in 2015. He is currently pursuing his Ph.D. degree in Mechanical
  Engineering from University of California, San Diego, CA, USA. He
  has been the recipient of the Campus-wide Best Undergraduate Student
  Award in 2013 from Isfahan University of Technology and the
  Mechanical and Aerospace Engineering Recruitment Fellowship in 2014
  from the University of California, San Diego. His research interests
  include complex networks, brain networks, networked and distributed
  control systems, and differential privacy.
\end{IEEEbiography}

\begin{IEEEbiography}[{\includegraphics[width=1in,
    height=1.25in,clip,keepaspectratio]
    {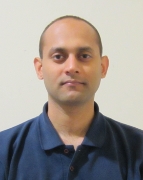}}]{Pavankumar Tallapragada}
  received the B.E. degree in Instrumentation Engineering from SGGS
  Institute of Engineering $\&$ Technology, Nanded, India in 2005,
  M.Sc. (Engg.) degree in Instrumentation from the Indian Institute of
  Science, Bangalore, India in 2007 and the Ph.D. degree in Mechanical
  Engineering from the University of Maryland, College Park in
  2013. He is currently a Postdoctoral Scholar in the Department of
  Mechanical and Aerospace Engineering at the University of
  California, San Diego. His research interests include
  event-triggered control, networked control systems, distributed
  control and networked transportation systems.
\end{IEEEbiography}

\begin{IEEEbiography}
  [{\includegraphics[width=1in,height=1.25in,clip,keepaspectratio]{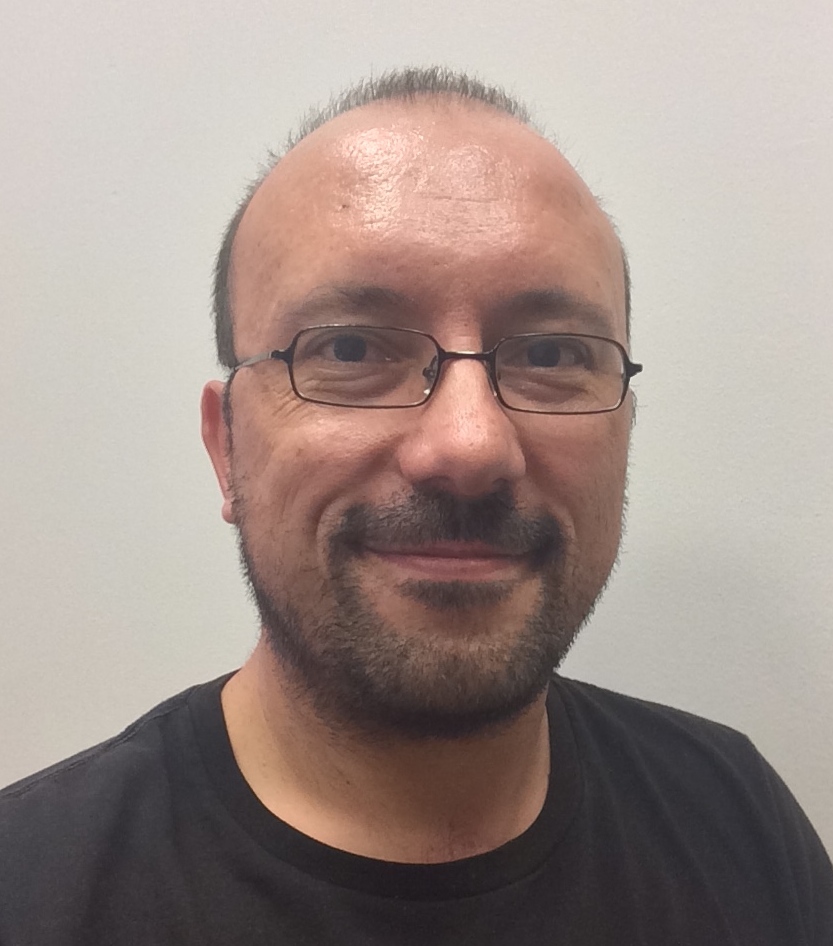}}]
  {Jorge Cort\'es} received the Licenciatura degree in mathematics
  from Universidad de Zaragoza, Zaragoza, Spain, in 1997, and the
  Ph.D. degree in engineering mathematics from Universidad Carlos III
  de Madrid, Madrid, Spain, in 2001. He held postdoctoral positions
  with the University of Twente, Twente, The Netherlands, and the
  University of Illinois at Urbana-Champaign, Urbana, IL, USA. He was
  an Assistant Professor with the Department of Applied Mathematics
  and Statistics, University of California, Santa Cruz, CA, USA, from
  2004 to 2007. He is currently a Professor in the Department of
  Mechanical and Aerospace Engineering, University of California, San
  Diego, CA, USA. He is the author of Geometric, Control and Numerical
  Aspects of Nonholonomic Systems (Springer-Verlag, 2002) and
  co-author (together with F. Bullo and S. Mart{\'\i}nez) of
  Distributed Control of Robotic Networks (Princeton University Press,
  2009). He is an IEEE Fellow and an IEEE Control Systems Society
  Distinguished Lecturer. His current research interests include
  distributed control, networked games, opportunistic state-triggered
  control and coordination, power networks, distributed optimization,
  spatial estimation, and geometric mechanics.
\end{IEEEbiography}

\end{document}